\documentclass[preprint,12pt]{elsarticle}
\usepackage{graphicx}
\usepackage{amssymb}
\usepackage{booktabs}

\usepackage{amsfonts}
\usepackage{amssymb}
\usepackage{amsmath}          % math
\usepackage{algorithm2e}
\usepackage{amsthm}
\theoremstyle{plain} \newtheorem{theorem}{Theorem}
\theoremstyle{plain} \newtheorem{definition}{Definition}
\theoremstyle{plain} \newtheorem{corollary}{Corollary}
\theoremstyle{plain} \newtheorem{lemma}{Lemma}
\theoremstyle{definition} \newtheorem{algor}{Algorithm}
\theoremstyle{definition}

\journal{}

\begin{document}

\begin{frontmatter}



\title{Accurate quotient-difference algorithm: error analysis, improvements and applications}


\author[a]{Peibing Du\fnref{work0}}\ead{dupeibing10@nudt.edu.cn}
\author[b]{Roberto Barrio\fnref{work1}} \ead{rbarrio@unizar.es}
\author[c]{Hao Jiang\corref{cor}\fnref{work2}}\ead{haojiang@nudt.edu.cn}
\author[a,d]{Lizhi Cheng\fnref{work3}}\ead{clzcheng@nudt.edu.cn}

\cortext[cor]{Corresponding author}

\fntext[work0]{Partially supported by National Natural Science Foundation of China (No. 61571008, No. 61401515).}
\fntext[work1]{Partially
supported by the Spanish Research project MTM2015-64095-P and by the European Social Fund and Diputaci\'on General de Arag\'on (Grant E48).}
\fntext[work2]{Partially supported by National Natural Science Foundation of China (No. 61402495, No. 61602166, No. 61303189, No. 61402496).}
\fntext[work3]{Partially supported by Science Project of National University of Defense Technology (JC120201) and National Natural Science Foundation of Hunan Province in China (13JJ2001).}

\address[a]{School of Science, National University of Defense Technology, Changsha, 410073, China}
\address[b]{Dpto. Matem\'atica Aplicada and IUMA, University of Zaragoza, E-50009 Zaragoza, Spain}
\address[c]{College of Computer, National University of Defense Technology, Changsha, 410073, China}
\address[d]{The State Key Laboratory of High  Performance Computation, National University of Defense Technology, Changsha, 410073, China}

\begin{abstract}
The compensated quotient-difference ({\tt Compqd}) algorithm is proposed along with some applications.
The main motivation is based on the fact that the standard quotient-difference ({\tt qd}) algorithm can be numerically unstable.
The {\tt Compqd} algorithm is obtained by applying error-free transformations to improve the traditional {\tt qd} algorithm.
We study in detail the error analysis of the {\tt qd} and {\tt Compqd} algorithms and we introduce new condition numbers so that the relative forward rounding error bounds
can be derived directly.  Our numerical experiments illustrate that the {\tt Compqd} algorithm is much more accurate than the {\tt qd} algorithm, relegating the influence of the condition numbers up to second order in the rounding unit of the computer. Three applications of the new algorithm in the obtention of continued fractions and in pole and zero detection are shown.
\end{abstract}


\end{frontmatter}

%%\linenumbers

%%------------------------------------------------------------------------Introduction-------------------------------------------------------------------------------------%%
\section{Introduction}\label{sec 1 intro}

The quotient-difference ({\tt qd}) algorithm was proposed by Rutishauser from previous works of Hadamard \cite{hadamard}, Aitken \cite{aitken1,aitken2}, and Lanczos \cite{lanczos} (for details see \cite{gutknecht}). This algorithm is highly related to the Pad\'e approximation \cite{lorentzen,CPVWJ,JT80} techniques. The {\tt qd} algorithm, and its variants, have  numerous applications. For instance, it  can be used to obtain the continuous fraction representation of meromorphic functions given by its power series development \cite{CPVWJ,JT80,Henrici}. It is  also related with complex analysis, as it provides a direct method to locate poles of complex functions \cite{Henrici,AC2010} and zeros of polynomials \cite{AC2010,rutishauser1}. Besides, in eigenvalue computation, the {\tt progressive qd} algorithm \cite{AC2010} has a relevant role as it can be interpreted as the LR transform for a tridiagonal matrix \cite{rutishauser2,parlett,fernando1994accurate}.

Unfortunately, in finite precision arithmetic, the quotient-difference algorithm has been shown in experiments to be numerically unstable. It is overly sensitive to rounding errors. As a consequence, high-precision arithmetic or exact arithmetic are recommended to overcome such a problem \cite{CUYT1997}.
In order to increase the accuracy and stability of algorithms for ill-conditioned problems, several researchers studied their corresponding accurate compensated algorithms by applying error-free transformations \cite{ORO,Rump081,Rump082} which can yield, in most circumstances, a full precision accuracy in standard precision. For instance, to evaluate ill-conditioned polynomials with floating-point coefficients, Graillat \emph{et al.} \cite{GLL1,GLL2,LL} proposed a compensated Horner algorithm to evaluate polynomials in monomial basis; Jiang \emph{et al.} \cite{JLCS,JBLLCS,DJC} presented compensated de-Casteljau and Clenshaw algorithms to evaluate polynomials in Bernstein, Chebyshev and Legendre basis, respectively.

In this paper, we first perform a complete analysis of the stability of the quotient-difference algorithm by providing forward rounding error bounds and we introduce condition numbers adapted to the problem that permit to give a simple error bound that helps to locate the instability problems. The bounds shown in this paper provide a theoretical statement of the numerical simulations in literature. To overcome, or at least, to delay the appearance of instability problems in standard precision, we introduce a new more accurate algorithm, the compensated quotient-difference algorithm. The proposed algorithm is based on error-free transformations. To obtain the compensated quotient-difference algorithm we consider, especially, the division operation in each inner loop which has never been used in previous works of compensated algorithms. Again, we perform a complete analysis of the stability and now, from the forward rounding error bounds, we observe that the condition numbers are multiplied by the square of the rounding unit, instead of the rounding unit. This result states that the proposed compensated quotient-difference algorithm is much more stable than the standard quotient-difference algorithm in working precision.

The paper is organized as follows. In Section \ref{sec2}, we introduce the classical {\tt qd} algorithm, some basic notations about floating-point arithmetic and error-free transformations. Section \ref{sec 3 err} presents the error analysis of the {\tt qd} algorithm and its condition numbers. In Section \ref{sec 4 comp}, the proposed new compensated {\tt qd} algorithm,  {\tt Compqd}, is provided. Section \ref{sec 5 comp err} presents the forward rounding error bounds of the {\tt Compqd} algorithm. Finally, in Section \ref{test}, we give several numerical experiments together with three practical applications to illustrate the efficiency, accuracy and stability of the new {\tt Compqd} algorithm. In the Appendices all the algorithms are detailed, and besides, a new compensated version of the \emph{progressive} form of the {\tt qd} scheme ({\tt Compproqd} algorithm) is given.

\section{Preliminaries}\label{sec2}
In this section we review the  classical {\tt qd} algorithm (Subsection \ref{subsec 2.1}). In order to perform the detailed error analysis of the algorithms, we give some basic notations (Subsection \ref{subsec 2.2}) and we present the error-free transformations (Subsection \ref{subsec 2.3}).

%----------------------------------------------------------------qd algorithm------------------------------------------------------------------
\subsection{The quotient-difference algorithm}\label{subsec 2.1}
Along this paper, quotient-difference is called {\tt qd} for short and we assume that the conditions for the existence of the {\tt qd} scheme (also known as the {\tt qd} table \cite{rutishauser}) are satisfied.

Considering the formal power series
\begin{equation}\label{FPS}
f(z)=c_0+c_1z+c_2z^2+\cdots \equiv \sum_{k=0}^{\infty}c_k z^k,
\end{equation}
where $c_i\in\mathbb{R}$,
we define its double sequence of Hankel determinants by

\begin{equation*}
\label{HDet}
H_m^{(n)}= \left|
\begin{array}{llll}
c_n & c_{n+1} & \cdots & c_{n+m}\\
c_{n+1} & c_{n+2} & \cdots & c_{n+m+1}\\
\vdots & \cdots & \cdots & \vdots\\
c_{n+m} & c_{n+m+1} & \cdots & c_{n+2m}
  \end{array}\right|, \qquad n, m \in \mathbb{N}.
  \end{equation*}
A remarkable connection among Hankel determinants \cite{CPVWJ} is given by
\begin{equation}\label{Hankel connection}
  (H_m^{(n)})^2+H_{m+1}^{(n-1)}H_{m-1}^{(n+1)}=H_{m}^{(n-1)}H_{m}^{(n+1)}.
\end{equation}
If we define
\begin{equation}\label{q and e}
  q_m^{(n)}=\frac{H_m^{(n+1)}H_{m-1}^{(n)}}{H_m^{(n)}H_{m-1}^{(n+1)}},~ ~ e_m^{(n)}=\frac{H_{m+1}^{(n)}H_{m-1}^{(n+1)}}{H_m^{(n)}H_{m}^{(n+1)}},
\end{equation}
then the previous relationship (\ref{Hankel connection}) can be interpreted as the following addition rhombus  rule
\begin{equation}\label{add q and e}
  q_m^{(n)}+e_m^{(n)}=q_m^{(n+1)}+e_{m-1}^{(n+1)},
\end{equation}
and, considering the definition (\ref{q and e}), $q_m^{(n)}$ and $e_m^{(n)}$ give the product rhombus rule
\begin{equation}\label{product q and e}
  q_m^{(n+1)}e_m^{(n+1)}=q_{m+1}^{(n)}e_{m}^{(n)}.
\end{equation}

\noindent Hence, both rhombus relations, (\ref{add q and e}) and (\ref{product q and e}),  give rise to the classical {\tt qd}  algorithm:\\
\vspace{0.2cm}

\hrule \vspace*{-0.2cm}
\begin{algor}{\tt qd} \label{Algor qd}\\
\indent~~{\bf input :} {$e_0^{(n)}=0$, $n=1,2,...$; $q_1^{(n)}=\frac{c_{n+1}}{c_n}$, $n=0,1,...$}\\
\indent~~{\bf output :}~~{qd scheme}\\
\indent~~{\bf for} ~{$m= 1, 2, ...$}\\
\indent~~\indent{\bf for} {$n= 0, 1, ...$}\\
 \indent~~\indent~~~~ $e_m^{(n)}=q_m^{(n+1)}-q_m^{(n)}+e_{m-1}^{(n+1)}$\\
\indent~~\indent~~~~  $q_{m+1}^{(n)}=({e_{m}^{(n+1)}}/{e_{m}^{(n)}}) \times q_{m}^{(n+1)}$\\
\indent~~\indent{\bf end}\\
\indent~~{\bf end}
\end{algor} \vspace*{-0.2cm}
\hrule \vspace{0.4cm}

The way of computing in Algorithm \ref{Algor qd} is explained in
the following {\tt qd} table showing the data connection via the above two rhombus rules from the first \emph{q}-column moving right.
\begin{equation}\label{qd table}
\begin{array}{cccccc}
     &q_1^{(0)} & & & & \\
    0& & e_1^{(0)}& & & \\
     &q_1^{(1)} & &q_2^{(0)} & & \\
     0& & e_1^{(1)}& &e_2^{(0)} & \\
      &q_1^{(2)} & &q_2^{(1)} & &q_3^{(0)} \\
      0& & e_1^{(2)}& &e_2^{(1)} & \\
      &\vdots & &q_2^{(2)} & &\vdots \\
      & & \vdots & &\vdots & \\
  \end{array}
\end{equation}

%----------------------------------------------------------------Notations------------------------------------------------------------------
\subsection{Basic notations}\label{subsec 2.2}

In this paper we assume to work with floating-point arithmetics adhering to IEEE-754 floating-point standard rounding to nearest. In our analysis we assume that there is no computational overflow or underflow. Let $op\in\{\oplus,\ominus,\otimes,\oslash\}$ represents a floating-point computation, and the evaluation of an expression in floating-point arithmetic is denoted $fl(\cdot)$, then its computation obeys the model
\begin{equation}\label{floating-point model}
a~op~b=fl(a\circ{b})=(a\circ{b})(1+\varepsilon_1)=(a\circ{b})/(1+\varepsilon_2),
\end{equation}
where $a,b\in\mathbb{F}$ (the set of floating-point numbers), $\circ\in\{+,-,\times,\div\}$ and $|\varepsilon_1|,|\varepsilon_2|\leq{u}$ ($u$ is the rounding unit of the computer).

For the following error analysis, let $v,x,y,z\in\mathbb{R}$ and
\begin{equation*}
\begin{split}
v&=x\circ{y}\circ{z},\\
\widehat{v}&=fl(\widehat{x}\circ\widehat{y}\circ\widehat{z}),\\
\widetilde{v}&=\widehat{x}\circ\widehat{y}\circ\widehat{z}.
\end{split}
\end{equation*}
Here, $\widehat{x}=fl(x)$, $\widehat{y}=fl(y)$, $\widehat{z}=fl(z)$, and $\widehat{x},\widehat{y},\widehat{z}\in\mathbb{F}$. The second equation can be rewritten as $\widehat{v}=\widehat{x}~op~\widehat{y}~op~\widehat{z}$.
We list some notations in Table \ref{notations} which will be helpful to understand this paper.

\begin{table}[ht]
\centering
\caption{Some notations for error analysis}\label{notations}
\begin{tabular}{c|c|c}\hline
Notations & Description & Equation \\\hline\hline
$e_m^{(n)}$ & the exact value & (\ref{theoretical e qd}) \\\hline
$\widetilde{e}_m^{(n)}$ & the result with the perturbed inputs in real arithmetic & (\ref{absolute perturbation definition 1}), (\ref{theoretical e qd}) \\\hline
$\widehat{e}_m^{(n)}$ & the result computed in floating-point arithmetic& (\ref{floating eq qd}), (\ref{absolute perturbation definition 2}) \\\hline
$\bar{e}_m^{(n)}$ & the proposed condition number & (\ref{abs e}) \\\hline
$\epsilon e_m^{(n)}$ & the compensated term of $\widehat{e}_m^{(n)}$& (\ref{absolute perturbation definition 1}), (\ref{step1 compqd}) \\\hline
$\widetilde{\epsilon e}_m^{(n)}$ & the perturbation for $\widetilde{e}_m^{(n)}$ & (\ref{absolute perturbation definition 2}) \\\hline
$\widehat{\epsilon e}_m^{(n)}$ & the approximate compensated term for $\widehat{e}_m^{(n)}$ & (\ref{approx step1 compqd}), (\ref{per of compensated term}) \\\hline
$\widetilde{\epsilon e}_m^{(n)*}$ & the approximate perturbation for $\widetilde{e}_m^{(n)}$ & (\ref{real step1 compqd}) \\\hline
$\epsilon\epsilon e_m^{(n)}$ & the compensated term of $\widehat{\epsilon e}_m^{(n)}$ & (\ref{per of compensated term}) \\\hline
\end{tabular}
\end{table}

The following definition and properties will also be used in error analysis (see more details in \cite{H}).
\begin{definition}\label{def1}
We define
\begin{equation*}
1+\theta_n=\prod^n_{i=1}(1+\delta_i)^{\rho_i},
\end{equation*}
where $|\delta_i|\leq{u},\rho_i=\pm1$ for  $i=1,2,\ldots,n$, $|\theta_n|\leq{\gamma_n}:=\dfrac{nu}{1-nu}=nu+\mathcal{O}(u^2)$ and $nu<1$.
\end{definition}

Other basic properties which will also be used in error analysis are given by:
\begin{itemize}
  \item $u+\gamma_k\leq\gamma_{k+1},$
  \item $i\gamma_k<\gamma_{ik},$
  \item $\gamma_k+\gamma_j+\gamma_k\gamma_j\leq{\gamma_{k+j}}.$
\end{itemize}

%-----------------EFT------------------------------------
\subsection{Error-free transformations}\label{subsec 2.3}
The development of some families of more stable algorithms, which are called \emph{compensated algorithms}~\cite{rump}, is based on the paper \cite{ORO} about \emph{Error-Free Transformations} (EFT). For a pair of floating-point numbers $a,b\in\mathbb{F}$, when no underflow occurs, there exists a floating-point number $y$ satisfying $a \circ b=x+y$, where $x={\rm fl}(a \circ b)$ and $\circ{\in}\{+ , - , \times \}$. Then the transformation
$(a,b)\longrightarrow (x,y)$ is regarded as an EFT. For division, the corresponding EFT is constructed using the reminder, so its definition is slightly different (see below).
The EFT algorithms of the addition, product and division of two floating-point numbers used later in this paper are the {\tt TwoSum} algorithm \cite{Knuth98}, the {\tt TwoProd} algorithm  \cite{Dekker71} and the {\tt DivRem} algorithm  \cite{ N.Louvet08,M.Pichat93}, respectively (see Appendix A). The following two theorems exhibit the main properties of those algorithms.

\begin{theorem}\label{EFT dingli}\textnormal{\cite{ORO}}
For  $a, b \in \mathbb{F}$ and $x, y \in \mathbb{F}$, when no underflow occurs, {\tt FastTwoSum}, {\tt TwoSum} and
{\tt TwoProd} algorithms verify
\begin{align*}
&[x,y]={\tt FastTwoSum}(a,b), \quad x={\rm fl}(a+b), \quad x+y=a+b, \quad
|y|\leq  u|x|, \quad |y|\leq {u}|a+b|,\\
&[x,y]={\tt TwoSum}(a,b), \quad x={\rm fl}(a+b), \quad x+y=a+b, \quad
|y|\leq  u|x|, \quad |y|\leq {u}|a+b|,\\
&[x,y]={\tt TwoProd}(a,b), \quad x={\rm fl}(a \times b), \quad x+y=a \times
b, \quad |y|\leq u|x|, \quad |y|\leq {u}|a\times b|.
\end{align*}
\end{theorem}

\begin{theorem}\label{EFT div dingli}\textnormal{\cite{N.Louvet08}}
For  $a, b \in \mathbb{F}$ and $q, r \in \mathbb{F}$, when no underflow occurs, {\tt DivRem} algorithm verifies
\begin{equation*}
[q,r]={\tt DivRem}(a,b), \quad a=b\times q+r, \quad q=a\oslash b, \quad
|r|\leq  u|b\times q|, \quad |r|\leq {u}|a|.
\end{equation*}
\end{theorem}

%-%-----------------------------------------------------------------------Error analysis of qd-------------------------------------------------------------------------------------%%
\section{Error analysis of the {\tt qd} algorithm}\label{sec 3 err}
In order to perform the error analysis of the complete  {\tt qd} algorithm, we split the process in two parts.

We begin with the  error analysis for the following \emph{inner loop} of the {\tt qd} algorithm in floating-point arithmetic,  where bold characters mean the `outputs' and
the rest mean the `inputs':
\begin{equation}\label{inner loop part}
\begin{array}{cccc}
     &q_m^{(n)}& &  \\
    e_{m-1}^{(n+1)}& &\bf e_m^{(n)}& \\
     &q_m^{(n+1)} & &\bf q_{m+1}^{(n)}\\
     & & e_m^{(n+1)}&  \\
    \end{array}
\end{equation}
There are two steps in the inner loop. The output of the first step is $\bf e_m^{(n)}$, and its inputs are $e_{m-1}^{(n+1)}$, $q_m^{(n)}$ and $q_m^{(n+1)}$. In next step, the inputs are $q_m^{(n+1)}$, $\bf e_m^{(n)}$ and $e_m^{(n+1)}$, and the output is $\bf q_{m+1}^{(n)}$.

Based on the error analysis of Subsection \ref{subsec3.1} of the inner loop, we obtain the rounding error bounds of the complete {\tt qd} algorithm by using mathematical induction in Subsection \ref{subsec3.2}.

%------Error analysis for inner loop of qd------------------------------
\subsection{Error analysis for the inner loop of the {\tt qd} algorithm}\label{subsec3.1}

In the proof of the stability analysis,  we first consider the perturbations of the floating-point inputs
\begin{equation}\label{absolute perturbation definition 1}
\begin{split}
\widehat{e}_m^{(n)} &= e_m^{(n)}+ \epsilon e_m^{(n)},\\
\widehat{q}_m^{(n)} &= q_m^{(n)}+ \epsilon q_m^{(n)},
\end{split}
\end{equation}
in the inner loop of the {\tt qd} algorithm.
Let
\begin{equation}\label{absolute perturbation definition 2}
\begin{split}
\widetilde{e}_m^{(n)} &= e_m^{(n)}+ \widetilde{\epsilon e}_m^{(n)},\\
\widetilde{q}_m^{(n)} &= q_m^{(n)}+ \widetilde{\epsilon q}_m^{(n)}.
\end{split}
\end{equation}
Here,
\begin{equation}\label{theoretical e qd}
\begin{split}
e_m^{(n)}&=q_m^{(n+1)}-q_m^{(n)}+e_{m-1}^{(n+1)},\\
\widetilde{e}_m^{(n)}&=\widehat{q}_m^{(n+1)}-\widehat{q}_m^{(n)}+\widehat{e}_{m-1}^{(n+1)},
\end{split}
\end{equation}
and
\begin{equation}\label{theoretical q qd}
\begin{split}
q_{m+1}^{(n)}&=\frac{e_{m}^{(n+1)}}{e_{m}^{(n)}}\times q_{m}^{(n+1)},\\
\widetilde{q}_{m+1}^{(n)}&=\frac{\widehat{e}_{m}^{(n+1)}}{\widehat{e}_{m}^{(n)}}\times \widehat{q}_{m}^{(n+1)}.
\end{split}
\end{equation}
In Equations~(\ref{theoretical e qd}) and (\ref{theoretical q qd}), all the computations are performed using real arithmetic without rounding error.
However, if all the computations are performed in floating-point arithmetic, we have
\begin{equation}\label{floating eq qd}
\begin{split}
fl(\widetilde{e}_m^{(n)})=\widehat{e}_m^{(n)}&=\widehat{q}_m^{(n+1)}\ominus\widehat{q}_m^{(n)}\oplus\widehat{e}_{m-1}^{(n+1)},\\
fl(\widetilde{q}_{m+1}^{(n)})=\widehat{q}_{m+1}^{(n)}&=fl(fl(\widehat{e}_{m}^{(n+1)}/\widehat{e}_{m}^{(n)}) \times \widehat{q}_{m}^{(n+1)}).
\end{split}
\end{equation}

The following Lemma~\ref{absolute perturbation of qd_0} gives the absolute  perturbation bounds of the floating-point inputs (\ref{absolute perturbation definition 1}) for the  inner loop of the {\tt qd} algorithm.

\begin{lemma}\label{absolute perturbation of qd_0}
The absolute perturbation bounds in the  inner loop of the {\tt qd} algorithm, considering floating-point inputs in real arithmetic, are given by
\begin{equation}\label{absolute perturbation bound e_0}
|\widetilde{\epsilon e}_m^{(n)}|\leq|\epsilon q_m^{(n+1)}|+|\epsilon q_m^{(n)}|+|\epsilon e_{m-1}^{(n+1)}|,
\end{equation}
and
\begin{equation}\label{absolute perturbation bound q_0}
|\widetilde{\epsilon q}_{m+1}^{(n)}|\leq \bar{\alpha}_{m+1}^{(n)},
\end{equation}
where
\begin{equation}\label{abs alpha}
\bar{\alpha}_{m+1}^{(n)}={b_m^{(n)}}\times\frac{|q_m^{(n+1)}||\epsilon e_{m}^{(n+1)}|+|e_{m}^{(n+1)}||\epsilon q_m^{(n+1)}|+|q_{m+1}^{(n)}||\epsilon e_m^{(n)}|+|\epsilon q_{m}^{(n+1)}||\epsilon e_{m}^{(n+1)}|}{|e_{m}^{(n)}|},
\end{equation}
with
\begin{equation}\label{bmn}
 b_m^{(n)}=\big|\frac{e_m^{(n)}}{e_m^{(n)}+\epsilon e_m^{(n)}}\big|,
\end{equation}
assuming $e_m^{(n)}\neq0$ and $\frac{\epsilon e_m^{(n)}}{e_m^{(n)}}\neq-1$.
\end{lemma}

\begin{proof}
From (\ref{absolute perturbation definition 1}) and (\ref{absolute perturbation definition 2}), we obtain that
\begin{equation}\label{equation absolut error e}
\widetilde{\epsilon e}_m^{(n)}=\epsilon q_m^{(n+1)}-\epsilon q_m^{(n)}+\epsilon e_{m-1}^{(n+1)},
\end{equation}
which gives us the first bound (\ref{absolute perturbation bound e_0}).

Similarly, we have
\begin{equation}\label{equation absolut error q}
\widetilde{\epsilon q}_{m+1}^{(n)}=\frac{q_m^{(n+1)}\epsilon e_{m}^{(n+1)}+e_{m}^{(n+1)}\epsilon q_m^{(n+1)}-q_{m+1}^{(n)}\epsilon e_m^{(n)}+\epsilon q_{m}^{(n+1)}\epsilon e_{m}^{(n+1)}}{e_{m}^{(n)}+\epsilon e_{m}^{(n)}}.
\end{equation}
Finally, if $e_m^{(n)}\neq0$ and $\frac{\epsilon e_m^{(n)}}{e_m^{(n)}}\neq-1$, we obtain (\ref{absolute perturbation bound q_0}).
\end{proof}

Assuming there exist the uniform bounds $|\epsilon e_{m}^{(n)}|\leq \epsilon_{abs}$,  $|\epsilon q_{m}^{(n)}|\leq \epsilon_{abs}$,  $\forall~ m,n \in \mathbb{N}$,
from Lemma \ref{absolute perturbation of qd_0} we will have
\begin{equation}\label{absolute perturbation e}
|\widetilde{e}_m^{(n)}-e_m^{(n)}|\leq 3\epsilon_{abs},
\end{equation}
and
\begin{equation}\label{absolute perturbation  q}
|\widetilde{q}_{m+1}^{(n)}- q_{m+1}^{(n)}|\leq b_m^{(n)}\times\frac{|q_m^{(n+1)}|+|e_{m}^{(n+1)}|+|q_{m+1}^{(n)}|}{|e_{m}^{(n)}|} \epsilon_{abs}+ \mathcal{O}(\epsilon_{abs}^2),
\end{equation}
where $b_m^{(n)}=\big|\frac{e_m^{(n)}}{e_m^{(n)}+\epsilon e_m^{(n)}}\big|$ with $e_m^{(n)}\neq0$ and $\frac{\epsilon e_m^{(n)}}{e_m^{(n)}}\neq-1$. It is obvious to see that
$
\lim\limits_{\epsilon e_m^{(n)}/e_m^{(n)}\longrightarrow -1}b_m^{(n)}\rightarrow\infty.
$

Consider that, under certain restrictions, the {\tt qd} scheme is constructed with the \emph{q}-columns tending to the reciprocal value of the simple pole of isolated modulus, while the corresponding \emph{e}-columns  tending to zero (see Theorem \ref{poles}). That is, for some $m$ we have that
\begin{equation*}
\lim_{n\longrightarrow \infty}e_m^{(n)}=0.
\end{equation*}
\noindent Then,
\begin{equation*}
\lim\limits_{n\longrightarrow \infty} \frac{|q_m^{(n+1)}|+|e_{m}^{(n+1)}|+|q_{m+1}^{(n)}|}{|e_{m}^{(n)}|}\rightarrow\infty.
\end{equation*}
This suggests that small absolute perturbations can cause large absolute errors in the computation of $q_{m+1}^{(n)}$ in (\ref{absolute perturbation  q}), that is, $\epsilon q_{m+1}^{(n)}$ may be large. Hence, $\epsilon_{abs}$  may not be small enough. Then, even in the case the computation of $e_{m+1}^{(n)}$ is well conditioned, which is  similar to that of the absolute error bound in~(\ref{absolute perturbation e}), the absolute perturbation bound (\ref{absolute perturbation  q}) can be arbitrary large. So, the {\tt qd} algorithm, as described in Subsection \ref{subsec 2.1}, can be highly unstable. Thus, our main question in this paper is oriented to improving its accuracy in order to use the algorithm in a wider range of situations (see Sections~\ref{sec 4 comp} and \ref{sec 5 comp err}).

%----------------------Roundoff error for inner loop of qd----------------------------------------
We have performed the perturbation analysis of the {\tt qd} algorithm. However, in practical numerical computations, the perturbation of the numerical results not only comes from the perturbation of inputs but also the accumulation of rounding errors generated in the algorithm itself. Now we focus on the obtention of the rounding error bounds for computing $\widetilde{e}_m^{(n)}$ in floating-point arithmetic, assuming that the floating-point inputs $\widehat{q}_m^{(n)}, \widehat{e}_{m-1}^{(n+1)}$ and $\widehat{q}_m^{(n+1)}$ are known exactly. Similarly, the rounding error bounds for computing $\widetilde{q}_{m+1}^{(n)}$ in floating-point arithmetic, assuming that the floating-point inputs $\widehat{e}_m^{(n)}, \widehat{q}_m^{(n+1)}$ and $\widehat{e}_m^{(n+1)}$ are known exactly.

\begin{lemma}\label{rounding error bound from exact inputs qd}
Let $fl(\widetilde{e}_m^{(n)})=\widehat{e}_m^{(n)}$ and $fl(\widetilde{q}_{m+1}^{(n)})=\widehat{q}_{m+1}^{(n)}$ are computed in floating-point arithmetic in the inner loop, then
\begin{equation}\label{rounding error bound e}
|fl(\widetilde{e}_m^{(n)})-\widetilde{e}_m^{(n)}|\leq\gamma_2\big(|\widehat{q}_m^{(n+1)}|+|\widehat{q}_m^{(n)}|+|\widehat{e}_{m-1}^{(n+1)}|\big),
\end{equation}
and
\begin{equation}\label{rounding error bound q}
|fl(\widetilde{q}_{m+1}^{(n)})-\widetilde{q}_{m+1}^{(n)}|\leq\gamma_2|\widetilde{q}_{m+1}^{(n)}|.
\end{equation}
\end{lemma}

\begin{proof}
It can be directly obtained from (\ref{floating-point model}) and (\ref{floating eq qd}).
\end{proof}

%-------Roundoff error bounds for inner loop of qd---------------

From Lemma~\ref{absolute perturbation of qd_0} and Lemma~\ref{rounding error bound from exact inputs qd}, we can derive the rounding error bounds for the inner loop of the {\tt qd} algorithm in floating-point arithmetic.

\begin{lemma}\label{rounding error bound from perturbed inputs qd}
The rounding error bounds in the inner loop of the {\tt qd} algorithm, considering perturbed floating-point inputs, are given by
\begin{equation}\label{round error bound from perturbed e}
|fl({\widetilde{e}}_{m}^{(n)})-{e}_{m}^{(n)}|\leq \gamma_2\big(|{q}_m^{(n+1)}|+|{q}_m^{(n)}|+|{e}_{m-1}^{(n+1)}|\big)+(1+\gamma_2)\big(|\epsilon q_m^{(n+1)}|+|\epsilon q_m^{(n)}|+|\epsilon e_{m-1}^{(n+1)}|\big),
\end{equation}
and
\begin{equation}\label{round error bound from perturbed q}
|fl({\widetilde{q}}_{m+1}^{(n)})-{q}_{m+1}^{(n)}|\leq \gamma_2|{{q}}_{m+1}^{(n)}|+(1+\gamma_2)\bar{\alpha}_{m+1}^{(n)},
\end{equation}
where $\bar{\alpha}_{m+1}^{(n)}$ are defined in (\ref{abs alpha}), and
$\epsilon e_m^{(n)}$ and $\epsilon q_{m}^{(n)}$ are the perturbations of the inputs $\widehat{e}_m^{(n)}$ and $\widehat{q}_{m}^{(n)}$, respectively.
\end{lemma}

\begin{proof}
First, we have
\begin{equation}\label{round off error e pertur input}
|fl({\widetilde{e}}_{m}^{(n)})-{e}_{m}^{(n)}|\leq |fl({\widetilde{e}}_{m}^{(n)})-\widetilde{e}_{m}^{(n)}|+|\widetilde{e}_{m}^{(n)}-{e}_{m}^{(n)}|.
\end{equation}
Then by (\ref{absolute perturbation bound e_0}) in Lemma \ref{absolute perturbation of qd_0}, (\ref{rounding error bound e}) in Lemma \ref{rounding error bound from exact inputs qd},  and (\ref{absolute perturbation definition 1}), we obtain

\begin{equation*}
\begin{split}
|fl({\widetilde{e}}_{m}^{(n)})-{e}_{m}^{(n)}|&\leq \gamma_2 \big(|{\widehat{q}}_m^{(n+1)}|+|{\widehat{q}}_m^{(n)}|+|{\widehat{e}}_{m-1}^{(n+1)}|\big)+\big(|\epsilon q_m^{(n+1)}|+|\epsilon q_m^{(n)}|+|\epsilon e_{m-1}^{(n+1)}|\big)\\
&\leq \gamma_2\big(|{q}_m^{(n+1)}|+|{q}_m^{(n)}|+|{e}_{m-1}^{(n+1)}|\big)+(1+\gamma_2)\big(|\epsilon q_m^{(n+1)}|+|\epsilon q_m^{(n)}|+|\epsilon e_{m-1}^{(n+1)}|\big).
\end{split}
\end{equation*}
Finally, taking into account (\ref{absolute perturbation bound q_0}) in Lemma \ref{absolute perturbation of qd_0} and (\ref{rounding error bound q}) in Lemma \ref{rounding error bound from exact inputs qd}, we have
\begin{equation*}
\begin{split}
|fl({\widetilde{q}}_{m+1}^{(n)})-{q}_{m+1}^{(n)}|&\leq |fl({\widetilde{q}}_{m+1}^{(n)})-\widetilde{q}_{m+1}^{(n)}|+|\widetilde{q}_{m+1}^{(n)}-{q}_{m+1}^{(n)}|\\
&\leq  \gamma_2|{{q}}_{m+1}^{(n)}|+(1+\gamma_2)\bar{\alpha}_{m+1}^{(n)}.
\end{split}
\end{equation*}
\end{proof}

%----------Roundoff error bounds of qd----------------------
\subsection{Rounding error bounds of {\tt qd} algorithm}\label{subsec3.2}

The previous results give us technical lemmas that allow us to give the global rounding error bounds of the {\tt qd} algorithm by using mathematical induction.

\begin{theorem}\label{rounding error bounds of qd}
The absolute forward rounding error bounds for the {\tt qd} algorithm, in the real coefficients case ($c_i\in\mathbb{R}$), are given by
\begin{equation*}\label{qd err for em}
|fl(\widetilde{e}_m^{(n)})-e_m^{(n)}|\leq\bigg(\prod_{i=0}^{m-1}B_i\bigg)\times\gamma_{4m}|\bar{e}_m^{(n)}|,
\end{equation*}
and
\begin{equation*}\label{qd err for q(m+1)}
|fl(\widetilde{q}_{m+1}^{(n)})-q_{m+1}^{(n)}|\leq\bigg(\prod_{i=0}^{m}B_i\bigg)\times\gamma_{4m+2}|\bar{q}_{m+1}^{(n)}|,
\end{equation*}
where
\begin{equation}\label{biqd}
B_m=\max_n\{b_m^{(n)*}\}, \qquad b_m^{(n)*}=\max\{b_m^{(n)},1\},
\end{equation}
with $b_m^{(n)}$ defined in (\ref{bmn}), and
\begin{equation}\label{abs e}
\bar{e}_m^{(n)}=|\bar{q}_m^{(n+1)}|+|\bar{q}_m^{(n)}|+|\bar{e}_{m-1}^{(n+1)}|,
\end{equation}
\begin{equation}\label{abs q}
\bar{q}_{m+1}^{(n)}=\bigg(\frac{|\bar{e}_m^{(n+1)}|}{|e_m^{(n+1)}|}+\frac{|\bar{q}_m^{(n+1)}|}{|q_m^{(n+1)}|}+\frac{|\bar{e}_m^{(n)}|}{|e_m^{(n)}|}\bigg)|q_{m+1}^{(n)}|,
\end{equation}
supposing
$\prod\limits_{i=1}^{m}B_i\ll\frac{1}{u}$, and where the initial values are given by $b^{(n)}_0=1, \, \bar{q}_1^{(n)}=q_1^{(n)}, \,  \bar{e}_0^{(n)}=0$.
\end{theorem}

\begin{proof}
It is easy to see that $B_m\geq 1$, $B_m\geq b_m^{(n)}$, $|e_m^{(n)}|\leq|\bar{e}_m^{(n)}|$ and $|q_{m+1}^{(n)}|\leq|\bar{q}_{m+1}^{(n)}|$ for $\forall m,n\in \mathbb{N}$ in (\ref{abs e}) and (\ref{abs q}).

\textbf{Step 1}: When $m=0$, there is just one floating-point division step, i.e. $q_1^{(n)}=c_{n+1}/c_n$. As $c_i\in\mathbb{R}$ for $i=0,1,2,\ldots$,  then $fl(\widetilde{q}_1^{(n)})=fl(fl(c_{n+1})/fl(c_n))=q_1^{(n)}(1+\theta_3)$. Hence we have
\begin{equation}\label{qd err for q1}
|fl(\widetilde{q}_1^{(n)})-q_1^{(n)}|\leq\gamma_3|\bar{q}_1^{(n)}|.
\end{equation}
Here, $fl(\widetilde{q}_1^{(n)})=\widehat{q}_1^{(n)}$ will be the input for computing $fl(\widetilde{e}_1^{(n)})$ and $fl(\widetilde{q}_2^{(n)})$.

\textbf{Step 2}: For $m=1$, we consider the rounding error bounds of $fl(\widetilde{e}_1^{(n)})$ and $fl(\widetilde{q}_2^{(n)})$. As $e_0^{(n)}=0$, we have
\begin{equation}\label{qd err for e1}
|fl(\widetilde{e}_1^{(n)})-e_1^{(n)}|\leq \big(\gamma_1+\gamma_3(1+\gamma_1)\big)|\bar{e}_1^{(n)}|\leq\gamma_4|\bar{e}_1^{(n)}|,
\end{equation}
where $\bar{e}_1^{(n)}=|\bar{q}_1^{(n+1)}|+|\bar{q}_1^{(n)}|$. The output $fl(\widetilde{e}_1^{(n)})$ will be the input $\widehat{e}_1^{(n)}$ for computing $fl(\widetilde{q}_2^{(n)})$ and $fl(\widetilde{e}_2^{(n)})$.

Considering
 $b^{(n)}_1\leq B_1$ and $\displaystyle{\bar{q}_2^{(n)}=\bigg(\frac{|\bar{e}_1^{(n+1)}|}{|e_1^{(n+1)}|}+1+\frac{|\bar{e}_1^{(n)}|}{|e_1^{(n)}|}\bigg)|q_2^{(n)}|}$, from (\ref{product q and e}), (\ref{absolute perturbation definition 1}), (\ref{abs alpha}), (\ref{qd err for q1}) and (\ref{qd err for e1}), we have
\begin{equation}\label{per err for q2}
\begin{split}
|\bar{\alpha}_2^{(n)}|&\leq {b^{(n)}_1}\times\bigg(\gamma_4\frac{|\bar{e}_1^{(n+1)}|}{|e_1^{(n+1)}|}|q_2^{(n)}|
+\gamma_3|q_2^{(n)}|+\gamma_4\frac{|\bar{e}_1^{(n)}|}{|e_1^{(n)}|}|q_2^{(n)}|+\gamma_3\gamma_4
\frac{|\bar{e}_1^{(n+1)}|}{|e_1^{(n+1)}|}|q_2^{(n)}|\bigg)\\
&\leq{B_1}\gamma_4|\bar{q}_2^{(n)}|.
\end{split}
\end{equation}
Hence, from (\ref{round error bound from perturbed q}) in Lemma~\ref{rounding error bound from perturbed inputs qd} and (\ref{per err for q2}), we derive
\begin{equation}\label{qd err for q2}
\begin{split}
|fl(\widetilde{q}_2^{(n)})-q_2^{(n)}|&\leq \bigg\{\gamma_2+(1+\gamma_2){B_1}\gamma_4\bigg(\frac{|\bar{e}_1^{(n+1)}|}{|e_1^{(n+1)}|}+1+\frac{|\bar{e}_1^{(n)}|}{|e_1^{(n)}|}\bigg)\bigg\}|q_2^{(n)}|\\
&\leq \big(\gamma_2+{B_1}\gamma_4(1+\gamma_2)\big)|\bar{q}_2^{(n)}|\\
&\leq{B_1}\gamma_6|\bar{q}_2^{(n)}|.
\end{split}
\end{equation}
The output $fl(\widetilde{q}_2^{(n)})$ will also be the input $\widehat{q}_2^{(n)}$ for computing $fl(\widetilde{e}_2^{(n)})$ and $fl(\widetilde{q}_3^{(n)})$.

\textbf{Step 3}: For $m=2$, we consider the rounding error bounds of $fl(\widetilde{e}_2^{(n)})$ and $fl(\widetilde{q}_3^{(n)})$. According to (\ref{round error bound from perturbed e}) in Lemma \ref{rounding error bound from perturbed inputs qd}, using  (\ref{abs e}), (\ref{qd err for e1}),  and (\ref{qd err for q2}), we have
\begin{equation}\label{qd err for e2}
\begin{split}
|fl(\widetilde{e}_2^{(n)})-e_2^{(n)}|&\leq\gamma_2\big(|q_2^{(n+1)}|+|q_2^{(n)}|+|e_1^{(n+1)}|\big)+(1+\gamma_2){B_1}\gamma_6|\bar{e}_2^{(n)}|\\
&\leq \big(\gamma_2+(1+\gamma_2){B_1}\gamma_6 \big)|\bar{e}_2^{(n)} |\\
&\leq{B_1}\gamma_8|\bar{e}_2^{(n)}|,
\end{split}
\end{equation}
where $\bar{e}_2^{(n)}=|\bar{q}_2^{(n+1)}|+|\bar{q}_2^{(n)}|+|\bar{e}_1^{(n+1)}|$. The output $fl(\widetilde{e}_2^{(n)})$ will be the input $\widehat{e}_2^{(n)}$ for computing $fl(\widetilde{q}_3^{(n)})$ and $fl(\widetilde{e}_3^{(n)})$.

Using (\ref{product q and e}), (\ref{abs alpha}), (\ref{abs q}), (\ref{qd err for q2}) and (\ref{qd err for e2}), and considering $B_1\ll\frac{1}{u}$ and $B_2\geq b_2^{(n)}$, we have
\begin{equation}\label{per err for q3}
\begin{split}
|\bar{\alpha}_3^{(n)}|&\leq{b_2^{(n)}}\times\bigg({B_1}\gamma_8\frac{|\bar{e}_2^{(n+1)}|}{|e_2^{(n+1)}|}
|q_3^{(n)}|+{B_1}\gamma_6\frac{|\bar{q}_2^{(n+1)}|}{|q_2^{(n+1)}|}|q_3^{(n)}|+{B_1}\gamma_8
\frac{|\bar{e}_2^{(n)}|}{|e_2^{(n)}|}|q_3^{(n)}|\\
&+({B_1})^2\gamma_6\gamma_8
\frac{|\bar{q}_2^{(n+1)}|}{|q_2^{(n+1)}|}\frac{|\bar{e}_1^{(n+1)}|}{|e_1^{(n+1)}|}|q_3^{(n)}|\bigg)\leq{B_1}{B_2}\gamma_8|\bar{q}_3^{(n)}|,
\end{split}
\end{equation}
Hence, from (\ref{round error bound from perturbed q}) in Lemma \ref{rounding error bound from perturbed inputs qd} and (\ref{per err for q3}), we derive
\begin{equation*}\label{qd err for q3}
\begin{split}
|fl(\widetilde{q}_3^{(n)})-q_3^{(n)}|&\leq \bigg\{\gamma_2+(1+\gamma_2){B_1B_2}\gamma_8 \bigg(\frac{|\bar{e}_2^{(n+1)}|}{|e_2^{(n+1)}|}+\frac{|\bar{q}_2^{(n+1)}|}{|q_2^{(n+1)}|}+\frac{|\bar{e}_2^{(n)}|}{|e_2^{(n)}|}\bigg)\bigg\}|q_3^{(n)}|\\
&\leq \big(\gamma_2+{B_1B_2}\gamma_8(1+\gamma_2)\big)|\bar{q}_3^{(n)}|\\
&\leq{B_1B_2}\gamma_{10}|\bar{q}_3^{(n)}|.
\end{split}
\end{equation*}
Then, we have found the regular pattern of the rounding error bounds.

\textbf{Step 4}: Now, for a generic $k \in \mathbb{N}$, we assume that when $m=k$ the absolute forward rounding error bounds of {\tt qd} algorithm are satisfied
\begin{equation}\label{qd err for em22}
|fl(\widetilde{e}_k^{(n)})-e_k^{(n)}|\leq\bigg({\prod_{i=0}^{k-1}B_i}\bigg)\gamma_{4k}|\bar{e}_k^{(n)}|,
\end{equation}
and
\begin{equation}\label{qd err for qm+1}
|fl(\widetilde{q}_{k+1}^{(n)})-q_{k+1}^{(n)}|\leq\bigg({\prod_{i=0}^{k}B_i}\bigg)\gamma_{4k+2}|\bar{q}_{k+1}^{(n)}|.
\end{equation}

In a similar way, for $m=k+1$, considering the rounding error bound of $fl(\widetilde{e}_{k+1}^{(n)})$ with inputs $fl(\widetilde{e}_k^{(n)})$ and $fl(\widetilde{q}_{k+1}^{(n)})$ in (\ref{qd err for em22}) and (\ref{qd err for qm+1}), from (\ref{abs e}) and (\ref{round error bound from perturbed e}) in Lemma \ref{rounding error bound from perturbed inputs qd}, we can derive that
\begin{equation}\label{qd err for em+1}
\begin{split}
|fl(\widetilde{e}_{k+1}^{(n)})-e_{k+1}^{(n)}|&\leq\gamma_2\big(|q_{k+1}^{(n+1)}|+|q_{k+1}^{(n)}|+|e_k^{(n+1)}|\big)+(1+\gamma_2)\bigg({\prod_{i=0}^{k}B_i}\bigg)\gamma_{4k+2}|\bar{e}_{k+1}^{(n)}|\\
&\leq \bigg\{\gamma_2+(1+\gamma_2)\bigg({\prod_{i=0}^{k}B_i}\bigg)\gamma_{4k+2} \bigg\}|\bar{e}_{k+1}^{(n)} |\\
&\leq\bigg({\prod_{i=0}^{k}B_i}\bigg)\gamma_{4(k+1)}|\bar{e}_{k+1}^{(n)}|.
\end{split}
\end{equation}
The output $fl(\widetilde{e}_{k+1}^{(n)})$ will be the input $\widehat{e}_{k+1}^{(n)}$ for computing $fl(\widetilde{q}_{k+2}^{(n)})$.

Next, considering the rounding error bound of $fl(\widetilde{q}_{k+2}^{(n)})$ with inputs $fl(\widetilde{q}_{k+1}^{(n)})$ and $fl(\widetilde{e}_{k+1}^{(n)})$ in (\ref{qd err for qm+1}) and (\ref{qd err for em+1}), from (\ref{product q and e}), (\ref{abs alpha}) and (\ref{abs q}), with $b_{k+1}^{(n)}\leq B_{k+1}$, we have
\begin{equation}\label{per err for qm+2}
\begin{split}
|\bar{\alpha}_{k+2}^{(n)}|&\leq{b_{k+1}^{(n)}}\times\bigg\{\bigg({\prod_{i=0}^{k}B_i}\bigg)\gamma_{4(k+1)}\frac{|\bar{e}_{k+1}^{(n+1)}|}{|e_{k+1}^{(n+1)}|}|q_{k+2}^{(n)}|+\bigg({\prod_{i=0}^{k}B_i}\bigg)\gamma_{4k+2}\frac{|\bar{q}_{k+1}^{(n+1)}|}{|q_{k+1}^{(n+1)}|}|q_{k+2}^{(n)}|\\
&+\bigg({\prod_{i=0}^{k}B_i}\bigg)\gamma_{4(k+1)}\frac{|\bar{e}_{k+1}^{(n)}|}{|e_{k+1}^{(n)}|}|q_{k+2}^{(n)}|+\bigg({\prod_{i=0}^{k}B_i}\bigg)^2\gamma_{4k+2}\gamma_{4(k+1)}\frac{|\bar{q}_{k+1}^{(n+1)}|}{|q_{k+1}^{(n+1)}|}\frac{|\bar{e}_{k}^{(n+1)}|}{|e_{k}^{(n+1)}|}|q_{k+2}^{(n)}|\bigg\}\\
&\leq\bigg({\prod_{i=0}^{k+1}B_i}\bigg)\gamma_{4(k+1)}|\bar{q}_{k+2}^{(n)}|.
\end{split}
\end{equation}
Hence, from (\ref{round error bound from perturbed q}) in Lemma \ref{rounding error bound from perturbed inputs qd} and (\ref{per err for qm+2}), with $1 \leq B_{i}$, we derive
\begin{equation*}\label{qd err for qm+2}
\begin{split}
|fl(\widetilde{q}_{k+2}^{(n)})-q_{k+2}^{(n)}|&\leq \bigg\{\gamma_2+(1+\gamma_2)\bigg({\prod_{i=0}^{k+1}B_i}\bigg)\gamma_{4(k+1)} \bigg(\frac{|\bar{e}_{k+1}^{(n+1)}|}{|e_{k+1}^{(n+1)}|}+\frac{|\bar{q}_{k+1}^{(n+1)}|}{|q_{k+1}^{(n+1)}|}+\frac{|\bar{e}_{k+1}^{(n)}|}{|e_{k+1}^{(n)}|}\bigg)\bigg\}|q_{k+2}^{(n)}|\\
&\leq \bigg\{\gamma_2+\bigg({\prod_{i=0}^{k+1}B_i}\bigg)\gamma_{4(k+1)}(1+\gamma_2)\bigg\}|\bar{q}_{k+2}^{(n)}|\\
&\leq\bigg({\prod_{i=0}^{k+1}B_i}\bigg)\gamma_{4(k+1)+2}|\bar{q}_{k+2}^{(n)}|.
\end{split}
\end{equation*}
And therefore, by induction we obtain the result.
\end{proof}

In order to simplify all the analysis, we define new condition numbers for evaluating each $e_m^{(n)}$ and $q_{m+1}^{(n)}$ using the {\tt qd} algorithm.

\begin{definition}\label{condition number}
The \emph{condition numbers} for evaluating the terms $e_m^{n}$ and $q_{m+1}^{(n)}$ using the {\tt qd} algorithm are defined by
$${\tt cond\_e_m^{(n)}}=\frac{\bar{e}_m^{(n)}}{|e_m^{(n)}|},$$
and
$${\tt cond\_q_{m+1}^{(n)}}=\frac{\bar{q}_{m+1}^{(n)}}{|q_{m+1}^{(n)}|},$$
where $\bar{e}_m^{(n)}$ and $\bar{q}_{m+1}^{(n)}$ are defined in (\ref{abs e}) and (\ref{abs q}), respectively,
and where $\bar{q}_1^{(n)}=q_1^{(n)}, \,  \bar{e}_0^{(n)}=e_0^{(n)}=0$.
\end{definition}

\begin{figure}[h!]
\begin{center}
\includegraphics[width=0.5\textwidth]{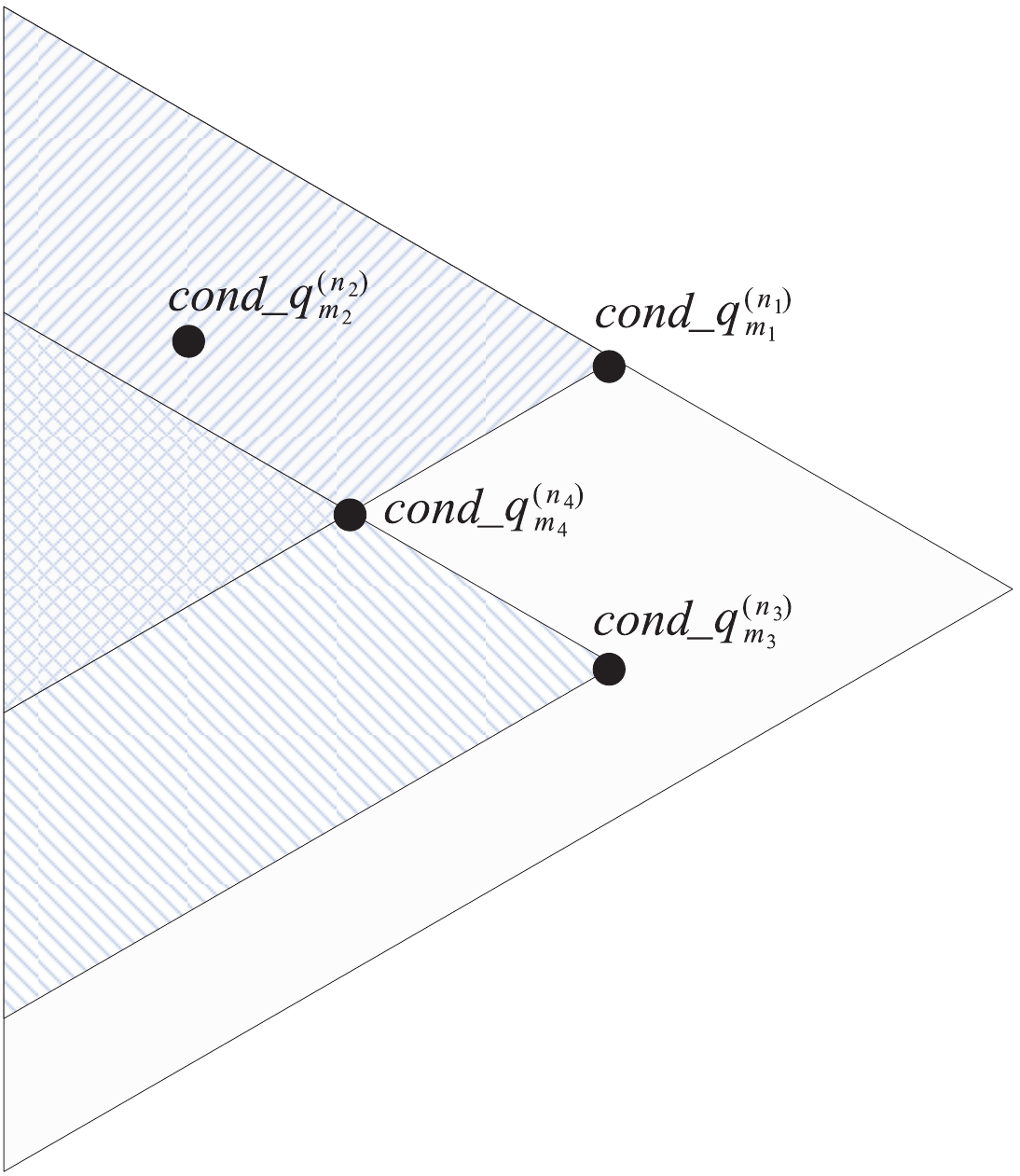}
\end{center}
\caption{Organization of the condition numbers in the {\tt qd} table.} \label{condqdtable}
\end{figure}

Therefore, each element in one {\tt qd} table in floating-point arithmetic has its own condition number. The relationship between two different elements is shown
in  Figure \ref{condqdtable}.  By Definition \ref{condition number}, from (\ref{abs q}) we have
$$\frac{\bar{q}_{m+1}^{(n)}}{|q_{m+1}^{(n)}|}\geq\frac{\bar{q}_m^{(n+1)}}{|q_m^{(n+1)}|}.$$
Hence, it is easy to see that
${\tt cond\_q_{{m_1}}^{({n_1})}}$ is larger than any element (e.g. ${\tt cond\_q_{{m_2}}^{({n_2})}}$) on its left part of the triangle, which corresponds to the terms of the {\tt qd} table's triangle that generates $q_{{m_1}}^{({n_1})}$.
It should be noticed that even though ${\tt cond\_q_{{m_1}}^{({n_1})}}\geq {\tt cond\_q_{{m_4}}^{({n_4})}}$ and ${\tt cond\_q_{{m_3}}^{({n_3})}}\geq {\tt cond\_q_{{m_4}}^{({n_4})}}$,
we can not say which one is larger between ${\tt cond\_q_{{m_1}}^{({n_1})}}$ and ${\tt cond\_q_{{m_3}}^{({n_3})}}$.
Next, we consider the condition number  ${\tt cond\_e_m^{(n)}}$.
From~(\ref{abs q}), we have
$$\frac{\bar{q}_{m}^{(n+1)}}{|q_{m}^{(n+1)}|}\geq\frac{\bar{e}_{m-1}^{(n+1)}}{|e_{m-1}^{(n+1)}|},  \qquad \frac{\bar{q}_{m}^{(n)}}{|q_{m}^{(n)}|}\geq\frac{\bar{e}_{m-1}^{(n+1)}}{|e_{m-1}^{(n+1)}|}.$$
Then, by (\ref{abs e}), we have
$$\frac{\bar{e}_m^{(n)}}{|e_m^{(n)}|}=\frac{\bar{q}_m^{(n+1)}+\bar{q}_m^{(n)}+\bar{e}_{m-1}^{(n+1)}}{|e_m^{(n)}|}=\frac{\bar{e}_{m-1}^{(n+1)}}{|e_{m-1}^{(n+1)}|}\times \frac{|{q}_m^{(n+1)}|+|{q}_m^{(n)}|+|{e}_{m-1}^{(n+1)}|}{|e_m^{(n)}|} \geq \frac{\bar{e}_{m-1}^{(n+1)}}{|e_{m-1}^{(n+1)}|}.$$
That is, the condition number  ${\tt cond\_e_m^{(n)}}$ has the same relationship as that of ${\tt cond\_q_{m+1}^{(n)}}$.

Using Theorem~\ref{rounding error bounds of qd} and the condition numbers given in Definition~\ref{condition number}, we can write the relative forward rounding error bounds of the {\tt qd} algorithm with perturbed inputs in a direct way:
\begin{corollary}\label{relative error bounds of qd}
The relative forward rounding error bounds for the {\tt qd} algorithm, in the real coefficients case ($c_i\in\mathbb{R}$), are given by
\begin{equation*}\label{qd err bound e}
\frac{|fl(\widetilde{e}_m^{(n)})-e_m^{(n)}|}{|e_m^{(n)}|}\leq{\Phi}_{m-1} \, \gamma_{4m} \, {\tt cond\_e_m^{(n)}} \equiv {\Phi}_{m-1} \, \mathcal{O}(u) \, {\tt cond\_e_m^{(n)}},
\end{equation*}
and
\begin{equation*}\label{qd err bound q}
\frac{|fl(\widetilde{q}_{m+1}^{(n)})-q_{m+1}^{(n)}|}{|q_{m+1}^{(n)}|}\leq{\Phi}_m \, \gamma_{4m+2} \, {\tt cond\_q_{m+1}^{(n)}} \equiv
{\Phi}_m \, \mathcal{O}(u) \, {\tt cond\_q_{m+1}^{(n)}},
\end{equation*}
supposing $\Phi_m={\prod\limits_{i=0}^{m}B_i} \ll\frac{1}{u}$ and where $B_i$ is defined in (\ref{biqd}).\end{corollary}

From Corollary \ref{relative error bounds of qd}, we can observe that ${\Phi}_{m-1}\leq {\Phi}_{m}$. Corollary~\ref{relative error bounds of qd} gives the theoretical analysis of the classic {\tt qd} algorithm. Now, our objective is to improve the error bounds by giving a new more stable algorithm.

%-%-----------------------------------------------------------------------compqd-------------------------------------------------------------------------------------%%

\section{Compensated {\tt qd} algorithm}\label{sec 4 comp}

In this section, we deduce the new compensated {\tt qd} algorithm.

Firstly, in order to consider the perturbations of the approximate inputs of the exact value $q_1^{(n)}=\frac{c_{n+1}}{c_n}$ in the {\tt qd} algorithm, we split each coefficient in the formal power series (\ref{FPS}), which is a real number, into three parts:
\begin{equation}\label{coefficients 3 part}
c_n=c^{(h)}_n+c^{(l)}_n+c^{(m)}_n,
\end{equation}
where $c_n,c^{(m)}_n\in\mathbb{R}$, $c^{(h)}_n,c^{(l)}_n\in\mathbb{F}$ and $|c^{(l)}_n|\leq{u}|c^{(h)}_n|$, $c^{(m)}_n$ is the remaining mantissa.
Referring to Table~\ref{notations}, we deem that using a double-double \cite{Bailey} number $(\widehat{q}_1^{(n)},-\widehat{\epsilon q}_1^{(n)})$ to approximate $q_1^{(n)}$, we can obtain more accurate initial values than $\widehat{q}_1^{(n)}=fl({c_{n+1}}/{c_n})$. Based on
\begin{equation*}\label{inputs approx}
\widehat{q}_1^{(n)}-\widehat{\epsilon q}_1^{(n)}\approx\frac{c^{(h)}_{n+1}+c^{(l)}_{n+1}}{c^{(h)}_n+c^{(l)}_n},
\end{equation*}
we utilize the double-double division arithmetic (Algorithm \ref{Div_dd_dd} in Appendix A) to get $(\widehat{q}_1^{(n)},-\widehat{\epsilon q}_1^{(n)})$.
Then, by using $u^2$ instead of $u$ (the approximate rounding unit in double-double arithmetic \cite{Bailey}) in Definition~\ref{def1}, from $\frac{3u^2}{1-3u^2}\leq\gamma_2^2$, we have
\begin{equation}\label{inputs}
|\widehat{q}_1^{(n)}-\widehat{\epsilon q}_1^{(n)}-q_1^{(n)}|\leq\gamma_2^2|q_1^{(n)}|,
\end{equation}
in double-double arithmetic.

Secondly, we deduce the compensated terms of outputs in each inner loop (\ref{inner loop part}) of the {\tt qd} algorithm.
In the inner loop of the {\tt qd} algorithm, the computations in floating-point arithmetic are present in (\ref{floating eq qd}).
By using EFTs, we can take into account the rounding errors generated on each operation and compensate them back to the original computed results to improve their accuracy:
\begin{equation}\label{EFT sum compqd}
[s,\mu_1]={\tt TwoSum}(\widehat{q}_m^{(n+1)},-\widehat{q}_m^{(n)}), \indent
[\widehat{e}_m^{(n)},\mu_2]={\tt TwoSum}(s,\widehat{e}_{m-1}^{(n+1)}),
\end{equation}
and
\begin{equation}\label{EFT pro compqd}
[t,\mu_3]={\tt DivRem}(\widehat{e}_{m}^{(n+1)},\widehat{e}_{m}^{(n)}),\indent
[\widehat{q}_{m+1}^{(n)},\mu_4]={\tt TwoProd}(t,\widehat{q}_{m}^{(n+1)}).
\end{equation}
By Theorem \ref{EFT dingli} and Theorem \ref{EFT div dingli}, we have
 \begin{equation}\label{firt term qd}
\begin{split}
s+\mu_1&=\widehat{q}_m^{(n+1)}-\widehat{q}_m^{(n)},\\
\widehat{e}_m^{(n)}+\mu_2&=s+\widehat{e}_{m-1}^{(n+1)},
\end{split}
\end{equation}
and
 \begin{equation}\label{second term qd}
\begin{split}
t\times\widehat{e}_{m}^{(n)}+\mu_3&=\widehat{e}_{m}^{(n+1)},\\
\widehat{q}_{m+1}^{(n)}+\mu_4&=t\times\widehat{q}_{m}^{(n+1)}.
\end{split}
\end{equation}
Computing $\widehat{e}_m^{(n)}$ with the perturbed inputs, from (\ref{add q and e}), (\ref{absolute perturbation definition 1}) and (\ref{firt term qd}), we can easily obtain the compensated term of $\widehat{e}_m^{(n)}$, given by
 \begin{equation}\label{step1 compqd}
\epsilon e_m^{(n)}=\epsilon q_m^{(n+1)}-\epsilon q_m^{(n)}+\epsilon e_{m-1}^{(n+1)}-\mu_1-\mu_2.
\end{equation}
Therefore, we can obtain the approximate compensated term of $\widehat{e}_m^{(n)}$ in floating-point arithmetic as
\begin{equation}\label{approx step1 compqd}
\widehat{\epsilon e}_m^{(n)}  \approx \widehat{\epsilon q}_m^{(n+1)}\ominus\widehat{\epsilon q}_m^{(n)}\oplus\widehat{\epsilon e}_{m-1}^{(n+1)}\ominus\mu_1\ominus\mu_2.
\end{equation}
When computing $\widehat{q}_{m+1}^{(n)}$ with perturbed inputs, from (\ref{second term qd}) we obtain that
 \begin{equation*}
\widehat{q}_{m+1}^{(n)}\widehat{e}_m^{(n)}+\mu_3\widehat{q}_m^{(n+1)}+\mu_4\widehat{e}_m^{(n)}=\widehat{q}_{m}^{(n+1)}\widehat{e}_m^{(n+1)},
\end{equation*}
then by (\ref{product q and e}) and (\ref{absolute perturbation definition 1}), we have
 \begin{equation*}
e_m^{(n)}\epsilon q_{m+1}^{(n)}+\epsilon e_m^{(n)} \widehat{q}_{m+1}^{(n)}+\mu_3\widehat{q}_m^{(n+1)}+\mu_4\widehat{e}_m^{(n)}=\epsilon q_m^{(n+1)} {e}_m^{(n+1)}+\epsilon e_m^{(n+1)} \widehat{q}_m^{(n+1)}.
\end{equation*}
Therefore, we obtain the compensated term of $\widehat{q}_{m+1}^{(n)}$, that is
 \begin{equation}\label{step2 compqd}
\epsilon q_{m+1}^{(n)}=\big(\epsilon q_m^{(n+1)}{e}_m^{(n+1)}+\epsilon e_m^{(n+1)} \widehat{q}_m^{(n+1)}-\epsilon e_m^{(n)} \widehat{q}_{m+1}^{(n)}-\mu_3\widehat{q}_m^{(n+1)}-\mu_4\widehat{e}_m^{(n)}\big)/{e}_m^{(n)}.
\end{equation}
Hence, the approximate compensated term of $\widehat{q}_{m+1}^{(n)}$ in floating-point arithmetic can be obtained from
\begin{equation}\label{approx step2 compqd}
\widehat{\epsilon q}_{m+1}^{(n)} \approx \big(\widehat{\epsilon q}_m^{(n+1)}\otimes \widehat{e}_m^{(n+1)}\oplus\widehat{\epsilon e}_m^{(n+1)}\otimes \widehat{q}_m^{(n+1)}\ominus\widehat{\epsilon e}_m^{(n)}\otimes\widehat{q}_{m+1}^{(n)}\ominus\mu_3\otimes\widehat{q}_m^{(n+1)}\ominus\mu_4\otimes\widehat{e}_m^{(n)}\big)\oslash \widehat{e}_m^{(n)}.
\end{equation}

Since $\widehat{\epsilon e}_m^{(n)}-\widehat{\epsilon e}_m^{(n)}$ and ${\widehat{q}}_{m+1}^{(n)}-\widehat{\epsilon q}_{m+1}^{(n)}$ are more accurate than $\widehat{e}_m^{(n)}$ and ${\widehat{q}}_{m+1}^{(n)}$, respectively,
we use {\tt FastTwoSum} (see Algorithm \ref{FastTwoSum} in Appendix A) to update the computed values $\widehat{e}_m^{(n)}$ and ${\widehat{q}}_{m+1}^{(n)}$ in each inner loop (\ref{inner loop part}) of the {\tt qd} algorithm in floating-point arithmetic with the compensated terms $\widehat{\epsilon {e}}_m^{(n)}$ in (\ref{approx step1 compqd}) and $\widehat{\epsilon {q}}_{m+1}^{(n)}$ in (\ref{approx step2 compqd}). The updated results, the floating-point numbers rounding to working precision, are expected to be more accurate than the original results.
Based on the discussion above, we propose the new compensated {\tt qd} algorithm, {\tt Compqd} (Algorithm \ref{Algor compqd}), which improves the accuracy of the classical {\tt qd} algorithm with a reasonable increment in the CPU time.\\
\vspace{0.1cm}

\hrule \vspace*{-0.1cm}
\begin{algor}{\tt Compqd} \label{Algor compqd}\\
\indent~~{\bf input:}~~{$\widehat{e}_0^{(n)}=0$, $\epsilon e_0^{(n)}=0$,  $n=1,2,...$ \\
\indent\indent~~~~~~~~$[\widehat{q}_1^{(n)},-\widehat{\epsilon q}_1^{(n)}]={\tt Div\_dd\_dd}(c^{(h)}_{n+1},c^{(l)}_{n+1},c^{(h)}_n,c^{(l)}_n)$, $n=0,1,...$}\\
\indent~~{\bf output:}~~{qd scheme}\\
\indent~~{\bf for} ~{$m= 1, 2, ...$}\\
\indent~~\indent{\bf for} {$n= 0, 1, ...$}\\
 \indent~~\indent~~~~  $[s,\mu_1]={\tt TwoSum}(\widehat{q}_m^{(n+1)},-\widehat{q}_m^{(n)})$\\
 \indent~~\indent~~~~  $[\widehat{e}_m^{(n)},\mu_2]={\tt TwoSum}(s,\widehat{e}_{m-1}^{(n+1)})$\\
 \indent~~\indent~~~~  $ \widehat{\epsilon e}_m^{(n)}=\widehat{\epsilon q}_m^{(n+1)}\ominus\widehat{\epsilon q}_m^{(n)}\oplus\widehat{\epsilon e}_{m-1}^{(n+1)}\ominus\mu_1\ominus\mu_2$\\
  \indent~~\indent~~~~ $[\widehat{e}_m^{(n)},-\widehat{\epsilon e}_m^{(n)}]={\tt FastTwoSum}(\widehat{e}_m^{(n)},-\widehat{\epsilon e}_m^{(n)})$\\
  \indent~~\indent~~~~ $[t,\mu_3]={\tt DivRem}(\widehat{e}_{m}^{(n+1)},\widehat{e}_{m}^{(n)})$\\
  \indent~~\indent~~~~ $[\widehat{q}_{m+1}^{(n)},\mu_4]={\tt TwoProd}(t,\widehat{q}_{m}^{(n+1)})$\\
   \indent~~\indent~~~~ $ \widehat{\epsilon q}_{m+1}^{(n)}=\big(\widehat{\epsilon q}_m^{(n+1)}\otimes \widehat{e}_m^{(n+1)}\oplus\widehat{\epsilon e}_m^{(n+1)}\otimes \widehat{q}_m^{(n+1)}\ominus\widehat{\epsilon e}_m^{(n)}\otimes\widehat{q}_{m+1}^{(n)}\ominus\mu_3\otimes\widehat{q}_m^{(n+1)}\ominus\mu_4\otimes\widehat{e}_m^{(n)}\big)\oslash \widehat{e}_m^{(n)}$\\
  \indent~~\indent~~~~ $[\widehat{q}_{m+1}^{(n)},-\widehat{\epsilon q}_{m+1}^{(n)}]={\tt FastTwoSum}(\widehat{q}_{m+1}^{(n)},-\widehat{\epsilon q}_{m+1}^{(n)})$\\
\indent~~\indent{\bf end}\\
\indent~~{\bf end}
 \end{algor}\vspace*{-0.1cm}
\hrule \vspace{0.4cm}

\noindent The {\tt Compqd} algorithm requires 69 flops in the inner loop.

We remark that if $c_i\in\mathbb{F}$ for $i=0,1,2\ldots$, the inputs $\widehat{q}_1^{(n)}$ and $\widehat{\epsilon q}_1^{(n)}$ of {\tt Compqd} can be obtained with
\begin{equation}\label{floating inputs of compqd}
[\widehat{q}_1^{(n)},r]={\tt DivRem}(c_{n+1},c_n), \quad -\widehat{\epsilon q}_1^{(n)}=r\oslash c_n.
\end{equation}

Moreover, it must be noticed that $\widehat{e}_m^{(n)}$ and $\widehat{q}_{m+1}^{(n)}$ in the {\tt Compqd} algorithm are different from those in {\tt qd} algorithm in floating-point arithmetic, because here we use {\tt FastTwoSum} to update the computed values in each inner loop.

%-----------Roundoff error bounds of compqd-------------------------
\section{Error analysis of {\tt Compqd} algorithm}\label{sec 5 comp err}

In a similar way as the error analysis of {\tt qd} algorithm, we first present the error analysis for the following \emph{inner loop} of the {\tt Compqd} algorithm in floating-point arithmetic in Subsection \ref{subsec5.1},  in which bold characters mean the `outputs' and
the rest mean the `inputs':
\begin{equation}\label{inner loop part of compqd}
\begin{array}{cccc}
     &\widehat{q}_m^{(n)}, \widehat{\epsilon q}_m^{(n)}& &  \\
    \widehat{e}_{m-1}^{(n+1)}, \widehat{\epsilon e}_{m-1}^{(n+1)}& &\bf \widehat{e}_m^{(n)}, \widehat{\epsilon e}_m^{(n)}& \\
     &\widehat{q}_m^{(n+1)}, \widehat{\epsilon q}_m^{(n+1)} & &\bf \widehat{q}_{m+1}^{(n)} , \widehat{\epsilon q}_{m+1}^{(n)}\\
     & & \widehat{e}_m^{(n+1)}, \widehat{\epsilon e}_m^{(n+1)}&  \\
    \end{array}
\end{equation}
For details, the inputs of the first step in the inner loop (\ref{inner loop part of compqd}) are $\widehat{e}_{m-1}^{(n+1)}$, $\widehat{q}_m^{(n)}$, $\widehat{q}_m^{(n+1)}$, $\widehat{\epsilon e}_{m-1}^{(n+1)}$, $\widehat{\epsilon q}_m^{(n)}$ and $\widehat{\epsilon q}_m^{(n+1)}$, while the outputs are $\bf \widehat{e}_m^{(n)}$ and $\bf \widehat{\epsilon e}_m^{(n)}$. The outputs of the second step are $\bf \widehat{q}_{m+1}^{(n)}$ and $\bf \widehat{\epsilon q}_{m+1}^{(n)}$, while the inputs are $\widehat{q}_m^{(n+1)}$, $\bf \widehat{e}_m^{(n)}$, $\widehat{e}_m^{(n+1)}$, $\widehat{\epsilon q}_m^{(n+1)}$, $\bf \widehat{\epsilon e}_m^{(n)}$, and $\widehat{\epsilon e}_m^{(n+1)}$.

In Subsection \ref{subsec5.2}, the rounding error bounds of $\widehat{e}_m^{(n)}-\widehat{\epsilon e}_m^{(n)}$ and $\widehat{q}_{m+1}^{(n)}-\widehat{\epsilon q}_{m+1}^{(n)}$ from {\tt Compqd} are obtained by using mathematical induction. Then, we finally give the rounding error bounds of the terms $\widehat{e}_m^{(n)}$ and $\widehat{q}_{m+1}^{(n)}$ of {\tt Compqd} updated by {\tt FastTwoSum}.

In this section, we denote the perturbations of the approximate compensated terms $\widehat{\epsilon e}_m^{(n)}$ and $\widehat{\epsilon q}_{m+1}^{(n)}$ in (\ref{approx step1 compqd}) and (\ref{approx step2 compqd}) by $\epsilon\epsilon e_m^{(n)}$ and $\epsilon\epsilon q_{m+1}^{(n)}$, respectively, which satisfy
\begin{equation}\label{per of compensated term}
\begin{split}
\widehat{\epsilon e}_m^{(n)}&=\epsilon e_m^{(n)}+\epsilon\epsilon e_m^{(n)},\\
\widehat{\epsilon q}_{m+1}^{(n)}&=\epsilon q_{m+1}^{(n)}+\epsilon\epsilon q_{m+1}^{(n)},
\end{split}
\end{equation}
where $\epsilon e_m^{(n)}$ and $\epsilon q_{m+1}^{(n)}$ are defined in (\ref{step1 compqd}) and (\ref{step2 compqd}).  Just like
$\epsilon e_m^{(n)}$ and $\epsilon q_{m+1}^{(n)}$  are the compensated terms of $ \widehat{e}_m^{(n)}$ and $ \widehat{q}_{m+1}^{(n)}$,
$\epsilon\epsilon e_m^{(n)}$ and $\epsilon\epsilon q_{m+1}^{(n)}$  are the compensated terms of $ \widehat{\epsilon e}_m^{(n)}$ and $ \widehat{ \epsilon q}_{m+1}^{(n)}$.
Then, from  (\ref{absolute perturbation definition 1}) and (\ref{per of compensated term}), we have
\begin{equation}\label{important equation of compqd}
\begin{split}
e_m^{(n)}-\epsilon\epsilon e_m^{(n)}&=\widehat{e}_m^{(n)}-\widehat{\epsilon e}_m^{(n)},\\
q_{m+1}^{(n)}-\epsilon\epsilon q_{m+1}^{(n)}&=\widehat{q}_{m+1}^{(n)}-\widehat{\epsilon q}_{m+1}^{(n)}.
\end{split}
\end{equation}
Note that $\widehat{e}_m^{(n)}$ and $\widehat{q}_{m+1}^{(n)}$ in {\tt Compqd} are different from those in {\tt qd} algorithm in floating-point arithmetic, but the values of $\widehat{e}_m^{(n)}-\widehat{\epsilon e}_m^{(n)}$ and $\widehat{q}_{m+1}^{(n)}-\widehat{\epsilon q}_{m+1}^{(n)}$ have not been changed when we use {\tt FastTwoSum} in each inner loop due to Theorem \ref{EFT dingli}.

\subsection{Error analysis for the inner loop of the {\tt Compqd} algorithm}\label{subsec5.1}
Before giving the error analysis, we note that the inputs of {\tt Compqd} algorithm have been updated by using {\tt FastTwoSum}, but the outputs in this subsection are not updated.

We first consider the perturbations of the floating-point inputs for the inner loop of {\tt Compqd} in real arithmetic.
Let
\begin{equation}\label{real step1 compqd}
\widetilde{\epsilon e}_m^{(n)*}=\widehat{\epsilon q}_m^{(n+1)}-\widehat{\epsilon q}_m^{(n)}+\widehat{\epsilon e}_{m-1}^{(n+1)}-\mu_1-\mu_2,
\end{equation}
\begin{equation}\label{real step2 compqd}
\widetilde{\epsilon q}_{m+1}^{(n)*}=(\widehat{\epsilon q}_m^{(n+1)}\widehat{e}_m^{(n+1)}+\widehat{\epsilon e}_m^{(n+1)} \widehat{q}_m^{(n+1)}-\widehat{\epsilon e}_m^{(n)} \widehat{q}_{m+1}^{(n)}-\widehat{\epsilon q}_m^{(n+1)}\widehat{\epsilon e}_m^{(n+1)}-\mu_3\widehat{q}_m^{(n+1)}-\mu_4\widehat{e}_m^{(n)})/(\widehat{e}_m^{(n)}-\widehat{\epsilon e}_m^{(n)}),
\end{equation}
where it should be noticed that $\widetilde{\epsilon e}_m^{(n)*}$ and $\widetilde{\epsilon q}_{m+1}^{(n)*}$ are different from $\widetilde{\epsilon e}_m^{(n)}$ and $\widetilde{\epsilon q}_{m+1}^{(n)}$ in (\ref{equation absolut error e}) and (\ref{equation absolut error q}), respectively.

In the following Lemma \ref{absolute perturbation of qd}, we evaluate the distance between $\widetilde{\epsilon e}_m^{(n)*}$ and ${\epsilon e}_m^{(n)}$ and the one between $\widetilde{\epsilon q}_{m+1}^{(n)*}$ and
${\epsilon q}_{m+1}^{(n)}$.

\begin{lemma}\label{absolute perturbation of qd}
The bounds of $\widetilde{\epsilon e}_m^{(n)*}$ and $\widetilde{\epsilon q}_{m+1}^{(n)*}$  are given by
\begin{equation}\label{absolute perturbation bound e}
|\widetilde{\epsilon e}_m^{(n)*}-\epsilon e_m^{(n)}|\leq|\epsilon\epsilon q_m^{(n+1)}|+|\epsilon\epsilon q_m^{(n)}|+|\epsilon\epsilon e_{m-1}^{(n+1)}|,
\end{equation}
and
\begin{equation}\label{absolute perturbation bound q}
|\widetilde{\epsilon q}_{m+1}^{(n)*}-\epsilon q_{m+1}^{(n)}|\leq\bar{\beta}_{m+1}^{(n)},
\end{equation}
where
\begin{equation}\label{abs beta}
\bar{\beta}_{m+1}^{(n)}={d_m^{(n)}}\times\frac{|q_m^{(n+1)}||\epsilon\epsilon e_{m}^{(n+1)}|+|e_{m}^{(n+1)}||\epsilon\epsilon q_m^{(n+1)}|+|q_{m+1}^{(n)}||\epsilon\epsilon e_m^{(n)}|+|\epsilon\epsilon q_{m}^{(n+1)}||\epsilon\epsilon e_{m}^{(n+1)}|}{|e_{m}^{(n)}|},
\end{equation}
with
\begin{equation}\label{dmn}
 d_m^{(n)}=\big|\frac{e_m^{(n)}}{e_m^{(n)}-\epsilon\epsilon e_m^{(n)}}\big|,
 \end{equation}
and supposing $e_m^{(n)}\neq0$ and $e_m^{(n)}\neq\epsilon\epsilon e_m^{(n)}$.

\end{lemma}

\begin{proof}
Taking into account the {\tt Compqd} algorithm in real arithmetic, considering (\ref{step1 compqd}),  the first half of (\ref{per of compensated term}) and (\ref{real step1 compqd}), we obtain that
\begin{equation}\label{equation absolut error e22}
\widetilde{\epsilon e}_m^{(n)*}-\epsilon e_m^{(n)}=\epsilon\epsilon q_m^{(n+1)}-\epsilon\epsilon q_m^{(n)}+\epsilon\epsilon e_{m-1}^{(n+1)},
\end{equation}
which can directly give us the first bound (\ref{absolute perturbation bound e}).

Similarly, from (\ref{step2 compqd}), the second half of (\ref{per of compensated term}) and (\ref{real step2 compqd}), we have
\begin{equation*}\label{equation absolut error q22}
\widetilde{\epsilon q}_{m+1}^{(n)*}-\epsilon q_{m+1}^{(n)}=\frac{q_m^{(n+1)}\epsilon\epsilon e_{m}^{(n+1)}+e_{m}^{(n+1)}\epsilon\epsilon q_m^{(n+1)}-q_{m+1}^{(n)}\epsilon\epsilon e_m^{(n)}-\epsilon\epsilon q_{m}^{(n+1)}\epsilon\epsilon e_{m}^{(n+1)}}{e_{m}^{(n)}-\epsilon\epsilon e_m^{(n)}},
\end{equation*}
that gives us the second bound (\ref{absolute perturbation bound q}).
\end{proof}

Then, we focus on  the distance between $\widetilde{\epsilon e}_m^{(n)*}$ and ${\widehat{\epsilon e}}_m^{(n)}$ and the one between $\widetilde{\epsilon q}_{m+1}^{(n)*}$ and
${\widehat{\epsilon q}}_{m+1}^{(n)}$.

\begin{lemma}\label{rounding error bounds of corrected terms in compqd for inner step}
 The distance between $\widetilde{\epsilon e}_m^{(n)*}$ and ${\widehat{\epsilon e}}_m^{(n)}$ and the one between $\widetilde{\epsilon q}_{m+1}^{(n)*}$ and
${\widehat{\epsilon q}}_{m+1}^{(n)}$ are given by
\begin{equation}\label{rounding error of step1}
|\widehat{\epsilon e}_m^{(n)}-\widetilde{\epsilon e}_m^{(n)*}|\leq\gamma_3\gamma_4 \big(|{q}_m^{(n+1)}-\epsilon\epsilon q_m^{(n+1)}|+|{q}_m^{(n)}-\epsilon\epsilon q_m^{(n)}|+|{e}_{m-1}^{(n+1)}-\epsilon\epsilon e_{m-1}^{(n+1)}|\big),
\end{equation}
and
\begin{equation}\label{rounding error of step2}
|\widehat{\epsilon q}_{m+1}^{(n)}-\widetilde{\epsilon q}_{m+1}^{(n)*}|\leq\gamma_7\gamma_8\frac{|{q}_m^{(n+1)}-\epsilon\epsilon{q}_m^{(n+1)}||{e}_m^{(n+1)}-\epsilon\epsilon{e}_m^{(n+1)}|}{|{e}_{m}^{(n)}-{\epsilon\epsilon e}_m^{(n)}|},
\end{equation}
where $\widetilde{\epsilon e}_m^{(n)*}$ and $\widetilde{\epsilon q}_{m+1}^{(n)*}$ are defined in (\ref{real step1 compqd}) and (\ref{real step2 compqd}), respectively.
\end{lemma}

\begin{proof}
We consider the rounding error of the approximate compensated term $\widehat{\epsilon e}_m^{(n)}$ for computing $\widehat{e}_m^{(n)}$.
From (\ref{floating-point model}) and (\ref{approx step1 compqd}), we obtain
\begin{equation*}
\widehat{\epsilon e}_m^{(n)}=\widehat{\epsilon q}_m^{(n+1)}(1+\theta_4)-\widehat{\epsilon q}_m^{(n)}(1+\theta_4)+\widehat{\epsilon e}_{m-1}^{(n+1)}(1+\theta_3)-\mu_1(1+\theta_2)-\mu_2(1+\theta_1).
\end{equation*}
Then, from (\ref{real step1 compqd}), we derive that
\begin{equation}\label{round error of per}
|\widehat{\epsilon e}_m^{(n)}-\widetilde{\epsilon e}_m^{(n)*}|\leq\gamma_4 \big(|\widehat{\epsilon q}_m^{(n+1)}|+|\widehat{\epsilon q}_m^{(n)}|+|\widehat{\epsilon e}_{m-1}^{(n+1)}|+|\mu_1|+|\mu_2|\big).
\end{equation}
According to Theorem \ref{EFT dingli} and (\ref{EFT sum compqd}),  we have
\begin{equation}\label{uneq 1}
\begin{split}
|\mu_1|&\leq{u}|\widehat{q}_m^{(n+1)}-\widehat{q}_m^{(n)}|\leq{u}(|\widehat{q}_m^{(n+1)}|+|\widehat{q}_m^{(n)}|),\\
|\mu_2|&\leq{u}|(\widehat{q}_m^{(n+1)}-\widehat{q}_m^{(n)})(1+\theta)+\widehat{e}_{m-1}^{(n+1)}|\leq{\gamma_1}(|\widehat{q}_m^{(n+1)}|+|\widehat{q}_m^{(n)}|+|\widehat{e}_{m-1}^{(n+1)}|).
\end{split}
\end{equation}
In each inner loop of {\tt Compqd}, all inputs have been updated by {\tt FastTwoSum}. For instance,  by Theorem~\ref{EFT dingli}, there is $|\widehat{\epsilon q}_m^{(n+1)}|\leq u|\widehat{q}_m^{(n+1)}-\widehat{\epsilon q}_m^{(n+1)}|$, where $\widehat{\epsilon q}_m^{(n+1)}$ is the value updated. The same results come for $\widehat{\epsilon q}_m^{(n)}$ and $\widehat{\epsilon e}_{m-1}^{(n+1)}$.
Hence, from (\ref{important equation of compqd}) and (\ref{uneq 1}), taking into account (\ref{round error of per}) and $u+2\gamma_1\leq\gamma_3$, we obtain
\begin{equation*}
\begin{split}
|\widehat{\epsilon e}_m^{(n)}-\widetilde{\epsilon e}_m^{(n)*}|\leq&\gamma_4\big\{u\big(|\widehat{q}_m^{(n+1)}-\widehat{\epsilon q}_m^{(n+1)}|+|\widehat{q}_m^{(n)}-\widehat{\epsilon q}_m^{(n)}|+|\widehat{e}_{m-1}^{(n+1)}-\widehat{\epsilon e}_{m-1}^{(n+1)}|\big)+\\
& \gamma_1\big(|\widehat{q}_m^{(n+1)}-\widehat{\epsilon q}_m^{(n+1)}|+|\widehat{q}_m^{(n)}-\widehat{\epsilon q}_m^{(n)}|\big)+\gamma_1\big(|\widehat{q}_m^{(n+1)}-\widehat{\epsilon q}_m^{(n+1)}|+\\
&|\widehat{q}_m^{(n)}-\widehat{\epsilon q}_m^{(n)}|+|\widehat{e}_{m-1}^{(n+1)}-\widehat{\epsilon e}_{m-1}^{(n+1)}|\big)\big\}\\
\leq&\gamma_3\gamma_4\big(|\widehat{q}_m^{(n+1)}-\widehat{\epsilon q}_m^{(n+1)}|+|\widehat{q}_m^{(n)}-\widehat{\epsilon q}_m^{(n)}|+|\widehat{e}_{m-1}^{(n+1)}-\widehat{\epsilon e}_{m-1}^{(n+1)}|\big)\\
\leq&\gamma_3\gamma_4\big(|{q}_m^{(n+1)}-\epsilon\epsilon q_m^{(n+1)}|+|{q}_m^{(n)}-\epsilon\epsilon q_m^{(n)}|+|{e}_{m-1}^{(n+1)}-\epsilon\epsilon e_{m-1}^{(n+1)}|\big).
\end{split}
\end{equation*}

Next, we consider the distance between $\widetilde{\epsilon q}_{m+1}^{(n)*}$ and
${\widehat{\epsilon q}}_{m+1}^{(n)}$.
As just commented above, all inputs have been updated by {\tt FastTwoSum} in each inner loop of {\tt Compqd}. Thus,  $\widehat{e}_{m}^{(n)}=\widehat{e}_{m}^{(n)}\ominus\widehat{\epsilon e}_{m}^{(n)}=(\widehat{e}_{m}^{(n)}-\widehat{\epsilon e}_{m}^{(n)})(1+\theta_1)$.
Therefore, from (\ref{approx step2 compqd}), we have
\begin{equation*}
\begin{split}
%\displaystyle{
\widehat{\epsilon q}_{m+1}^{(n)}&=\{\widehat{\epsilon q}_m^{(n+1)}\widehat{e}_m^{(n+1)}(1+\theta_7)+\widehat{\epsilon e}_m^{(n+1)}\widehat{q}_m^{(n+1)}(1+\theta_7)-\widehat{\epsilon e}_m^{(n)}\widehat{q}_{m+1}^{(n)}(1+\theta_6)-\\
&\mu_3\widehat{q}_m^{(n+1)}(1+\theta_5)-\mu_4\widehat{e}_m^{(n)}(1+\theta_4)\}/(\widehat{e}_{m}^{(n)}-\widehat{\epsilon e}_{m}^{(n)}).
%}
\end{split}
\end{equation*}
Then, from (\ref{real step2 compqd}), we derive that
\begin{equation}\label{compneg 6}
\begin{split}
|\widehat{\epsilon q}_{m+1}^{(n)}-\widetilde{\epsilon q}_{m+1}^{(n)*}|\leq&\gamma_7\bigg(\frac{|\widehat{\epsilon q}_m^{(n+1)}||\widehat{e}_m^{(n+1)}|+|\widehat{\epsilon e}_m^{(n+1)}||\widehat{q}_m^{(n+1)}|+|\widehat{\epsilon e}_m^{(n)}||\widehat{q}_{m+1}^{(n)}|}{|\widehat{e}_{m}^{(n)}-\widehat{\epsilon e}_m^{(n)}|}\\
& \qquad +\frac{|\mu_3||\widehat{q}_m^{(n+1)}|}{|\widehat{e}_{m}^{(n)}-\widehat{\epsilon e}_m^{(n)}|}+\frac{|\mu_4||\widehat{e}_m^{(n)}|}{|\widehat{e}_{m}^{(n)}-\widehat{\epsilon e}_m^{(n)}|}\bigg)+\frac{|\widehat{\epsilon e}_m^{(n+1)}||\widehat{\epsilon q}_m^{(n+1)}|}{|\widehat{e}_{m}^{(n)}-\widehat{\epsilon e}_m^{(n)}|}.
\end{split}
\end{equation}
Here, we consider that the output $\widehat{q}_{m+1}^{(n)}$ is not updated by using {\tt FastTwoSum}. Hence, we have
\begin{equation}\label{qm+1 floating}
\widehat{q}_{m+1}^{(n)}=\widehat{e}_m^{(n+1)}\oslash\widehat{e}_m^{(n)}\otimes\widehat{q}_m^{(n+1)}=\widehat{e}_{m}^{(n+1)}/\widehat{e}_{m}^{(n)} \times\widehat {q}_{m}^{(n+1)}(1+\theta_2)\leq\widehat{e}_{m}^{(n+1)}/\widehat{e}_{m}^{(n)} \times\widehat {q}_{m}^{(n+1)}(1+\gamma_2).
\end{equation}
Then, by Theorems \ref{EFT dingli}, \ref{EFT div dingli} and (\ref{EFT pro compqd}), we have
\begin{equation}\label{uneq 3}
\begin{split}
|\mu_3|\leq&u|\widehat{e}_m^{(n+1)}|,\\
|\mu_4|\leq&u|\widehat{q}_{m+1}^{(n)}|.
\end{split}
\end{equation}
Similarly, from Theorem \ref{EFT dingli}, we have $|\widehat{\epsilon q}_m^{(n+1)}|\leq u|\widehat{q}_m^{(n+1)}-\widehat{\epsilon q}_m^{(n+1)}|$, $|\widehat{\epsilon e}_m^{(n+1)}|\leq u|\widehat{e}_m^{(n+1)}-\widehat{\epsilon e}_m^{(n+1)}|$, $|\widehat{q}_m^{(n+1)}|\leq \frac{1}{1-u}|\widehat{q}_m^{(n+1)}-\widehat{\epsilon q}_m^{(n+1)}|$ and $|\widehat{e}_m^{(n+1)}|\leq \frac{1}{1-u}|\widehat{e}_m^{(n+1)}-\widehat{\epsilon e}_m^{(n+1)}|$.
By (\ref{qm+1 floating}), we also have $|\widehat{q}_{m+1}^{(n)}||\widehat{e}_m^{(n)}|\leq(1+\gamma_2)|\widehat{q}_m^{(n+1)}||\widehat{e}_m^{(n+1)}|$.

Finally, from(\ref{important equation of compqd}) and (\ref{uneq 3}), taking into account (\ref{compneg 6}),  with $\gamma_6\gamma_7+\gamma_1^2\leq\gamma_7\gamma_8$, we obtain
\begin{equation*}
\begin{split}
|\widehat{\epsilon q}_{m+1}^{(n)}-\widetilde{\epsilon q}_{m+1}^{(n)*}|\leq&\gamma_7\bigg(\gamma _6\times\frac{|{q}_m^{(n+1)}-\epsilon\epsilon{q}_m^{(n+1)}||{e}_m^{(n+1)}-\epsilon\epsilon{e}_m^{(n+1)}|}{|{e}_{m}^{(n)}-{\epsilon\epsilon e}_m^{(n)}|}\bigg)\\
& \qquad +\gamma_1^2\frac{|{q}_m^{(n+1)}-\epsilon\epsilon{q}_m^{(n+1)}||{e}_m^{(n+1)}-\epsilon\epsilon{e}_m^{(n+1)}|}{|{e}_{m}^{(n)}-{\epsilon\epsilon e}_m^{(n)}|}\\
\leq&\gamma_7\gamma_8\frac{|{q}_m^{(n+1)}-\epsilon\epsilon{q}_m^{(n+1)}||{e}_m^{(n+1)}-\epsilon\epsilon{e}_m^{(n+1)}|}{|{e}_{m}^{(n)}-{\epsilon\epsilon e}_m^{(n)}|}.
\end{split}
\end{equation*}

\end{proof}

Now, we present the rounding error bounds from perturbed inputs in the inner loop (\ref{inner loop part of compqd}) of {\tt Compqd}.

\begin{lemma}\label{rounding error of comp for inner step}
The rounding error bounds for the inner loop (\ref{inner loop part of compqd}) of the {\tt Compqd} algorithm, considering perturbed inputs, are given by
\begin{equation}\label{compneq 3}
|\widehat{\epsilon e}_m^{(n)}-\epsilon e_m^{(n)}|\leq\gamma_3\gamma_4\big(|q_m^{(n+1)}|+|q_m^{(n)}|+|e_{m-1}^{(n+1)}|\big)+(1+\gamma_3\gamma_4)\big(|\epsilon\epsilon q_m^{(n+1)}|+|\epsilon\epsilon q_m^{(n)}|+|\epsilon\epsilon e_{m-1}^{(n+1)}|\big),
\end{equation}
and
\begin{equation}\label{compneq 9}
|\widehat{\epsilon q}_{m+1}^{(n)}-\epsilon q_{m+1}^{(n)}|\leq\gamma_7\gamma_8d_m^{(n)}|q_{m+1}^{(n)}|+(1+\gamma_7\gamma_8)\bar{\beta}_{m+1}^{(n)},
\end{equation}
where $\bar{\beta}_{m+1}^{(n)}$ is defined in (\ref{abs beta}).
\end{lemma}

\begin{proof}
From (\ref{per of compensated term}), we have
\begin{equation*}\label{rounding error of e compqd inner loop}
|\widehat{\epsilon e}_m^{(n)}-\epsilon e_m^{(n)}|\leq|\widehat{\epsilon e}_m^{(n)}-\widetilde{\epsilon e}_m^{(n)*}|+|\widetilde{\epsilon e}_m^{(n)*}-e_m^{(n)}|.
\end{equation*}
Hence, using (\ref{absolute perturbation bound e}) in Lemma \ref{absolute perturbation of qd} and (\ref{rounding error of step1}) in Lemma \ref{rounding error bounds of corrected terms in compqd for inner step}, we can derive the first rounding error bound (\ref{compneq 3}).

Next, we obtain that
\begin{equation*}\label{rounding error of q compqd inner loop}
|\widehat{\epsilon q}_{m+1}^{(n)}-\epsilon q_{m+1}^{(n)}|\leq|\widehat{\epsilon q}_{m+1}^{(n)}-\widetilde{\epsilon q}_{m+1}^{(n)*}|+|\widetilde{\epsilon q}_{m+1}^{(n)*}-q_{m+1}^{(n)}|.
\end{equation*}
From (\ref{product q and e}) and (\ref{rounding error of step2}) in Lemma \ref{rounding error bounds of corrected terms in compqd for inner step}, we have that
\begin{equation}\label{rounding error of q compqd known inputs}
|\widehat{\epsilon q}_{m+1}^{(n)}-\widetilde{\epsilon q}_{m+1}^{(n)*}|\leq{d}_m^{(n)}\times\gamma_7\gamma_8\bigg(|q_{m+1}^{(n)}|+\frac{|{q}_m^{(n+1)}||\epsilon\epsilon{e}_m^{(n+1)}|+|{e}_m^{(n+1)}||
\epsilon\epsilon{q}_m^{(n+1)}|+|\epsilon\epsilon{q}_m^{(n+1)}||\epsilon\epsilon{e}_m^{(n+1)}|}{|{e}_{m}^{(n)}|}\bigg),
\end{equation}
where $d_m^{(n)}=\big|\frac{e_m^{(n)}}{e_m^{(n)}-\epsilon\epsilon e_m^{(n)}}\big|$ with $e_m^{(n)}\neq0$ and $\frac{\epsilon\epsilon e_m^{(n)}}{e_m^{(n)}}\neq1$.
Then, taking into account (\ref{absolute perturbation bound q}) in Lemma \ref{absolute perturbation of qd} and (\ref{rounding error of q compqd known inputs}), we derive the second rounding error bound (\ref{compneq 9}).
\end{proof}

\subsection{Rounding error bounds of the {\tt Compqd} algorithm}\label{subsec5.2}

With the previous results, we proceed in a similar way as in Subsection \ref{subsec3.2}, obtaining the rounding error bounds of $\widehat{e}_m^{(n)}-\widehat{\epsilon e}_m^{(n)}$ and $\widehat{q}_{m+1}^{(n)}-\widehat{\epsilon q}_{m+1}^{(n)}$ by the {\tt Compqd} algorithm from the perturbed inputs using  mathematical induction in Theorem \ref{rounding error bounds of comp qd with 2}. Here, we see $\widehat{e}_m^{(n)}-\widehat{\epsilon e}_m^{(n)}$ as a number with high accuracy,  the same case comes with $\widehat{q}_{m+1}^{(n)}-\widehat{\epsilon q}_{m+1}^{(n)}$. After that, we study the rounding error bounds of $\widehat{e}_m^{(n)}$ and $\widehat{q}_{m+1}^{(n)}$ updated by {\tt FastTwoSum} which will be shown in Theorem \ref{rounding error bounds of comp qd}.

\begin{theorem}\label{rounding error bounds of comp qd with 2}
The forward rounding error bounds of $\widehat{e}_m^{(n)}-\widehat{\epsilon e}_m^{(n)}$ and $\widehat{q}_{m+1}^{(n)}-\widehat{\epsilon q}_{m+1}^{(n)}$ from the {\tt Compqd} algorithm, in the real coefficients case ($c_i\in\mathbb{R}$), are given by
\begin{equation}\label{compqd err for em with 2}
|\widehat{e}_m^{(n)}-\widehat{\epsilon e}_m^{(n)}-e_m^{(n)}|\leq\bigg(\prod_{i=0}^{m-1}D_i\bigg)\times\gamma_{11m-5}\gamma_{11m-4}|\bar{e}_m^{(n)}|,
\end{equation}
and
\begin{equation}\label{compqd err for q(m+1) with 2}
|\widehat{q}_{m+1}^{(n)}-\widehat{\epsilon q}_{m+1}^{(n)}-q_{m+1}^{(n)}|\leq\bigg(\prod_{i=0}^{m}D_i\bigg)\times\gamma_{11m+2}^2|\bar{q}_{m+1}^{(n)}|,
\end{equation}
with
\begin{equation}\label{def D}
D_m=\max\limits_n\big\{d_m^{(n)*}\big\}, \qquad
d_m^{(n)*}=\max\big\{d_m^{(n)},1\big\},
\end{equation}
and where $d_m^{(n)}$ is defined in (\ref{dmn}), $\bar{e}_m^{(n)}$ and $\bar{q}_{m+1}^{(n)}$ are defined in Definition \ref{condition number}, and supposing  $\prod\limits_{i=1}^{m}D_i\ll\frac{1}{u}$, $e_m^{(n)}\neq0$, $\frac{\epsilon\epsilon e_m^{(n)}}{e_m^{(n)}}\neq1$. The initial values are given by $d_0^{(n)}=1, \, \bar{q}_1^{(n)}=q_1^{(n)}, \,  \bar{e}_0^{(n)}=0$.
\end{theorem}

\begin{proof}
It is easy to see that $D_m\geq 1$, $D_m\geq d_m^{(n)}$, $|e_m^{(n)}|\leq|\bar{e}_m^{(n)}|$ and $|q_{m+1}^{(n)}|\leq|\bar{q}_{m+1}^{(n)}|$,  $\forall m, n\in \mathbb{N}$.

\textbf{Step 1}: When $m=0$, we consider the perturbations of the inputs $\widehat{q}_1^{(n)}-\widehat{\epsilon q}_1^{(n)}$. From (\ref{inputs}) and (\ref{important equation of compqd}), we obtain
\begin{equation}\label{compqd err for q1}
|\epsilon\epsilon q_1^{(n)}|=|\widehat{q}_1^{(n)}-\widehat{\epsilon q}_1^{(n)}-q_1^{(n)}|\leq\gamma_2^2|q_1^{(n)}|\leq\gamma_2^2|\bar{q}_1^{(n)}|.
\end{equation}

\textbf{Step 2}: For $m=1$, as $e_0^{(n)}=0$, according to (\ref{compneq 3}) in Lemma \ref{rounding error of comp for inner step} and (\ref{compqd err for q1}), from (\ref{per of compensated term}) we have
\begin{equation*}\label{compqd err for e1}
|\epsilon\epsilon e_1^{(n)}|\leq\big(\gamma_3\gamma_4+\gamma_2^2(1+\gamma_3\gamma_4)\big)|\bar{e}_1^{(n)}|\leq\gamma_5\gamma_6|\bar{e}_1^{(n)}|,
\end{equation*}
where $\bar{e}_1^{(n)}=|q_1^{(n+1)}|+|q_1^{(n)}|$.

Besides, considering $\bar{\beta}_2^{(n)}$ in Lemma \ref{rounding error of comp for inner step} which is defined in (\ref{abs beta}), we obtain
\begin{equation}\label{compqd per err for q2}
\begin{split}
|\bar{\beta}_2^{(n)}|&\leq{d_1^{(n)}}\times\bigg(\gamma_5\gamma_6\frac{|\bar{e}_1^{(n+1)}|}{|e_1^{(n+1)}|}|q_2^{(n)}|+\gamma_2^2|q_2^{(n)}|+\gamma_5\gamma_6\frac{|\bar{e}_1^{(n)}|}{|e_1^{(n)}|}|q_2^{(n)}|+\gamma_2^2\gamma_5\gamma_6\frac{|\bar{q}_1^{(n+1)}|}{|q_1^{(n+1)}|}\frac{|\bar{e}_1^{(n+1)}|}{|e_1^{(n+1)}|}|q_2^{(n)}|\bigg)\\
&\leq{D_1}\gamma_5\gamma_6|\bar{q}_2^{(n)}|,
\end{split}
\end{equation}
with $\displaystyle{\bar{q}_2^{(n)}=\bigg(\frac{|\bar{e}_1^{(n+1)}|}{|e_1^{(n+1)}|}+1+\frac{|\bar{e}_1^{(n)}|}{|e_1^{(n)}|}\bigg)|q_2^{(n)}|}$. Hence, from (\ref{per of compensated term}), (\ref{compneq 9}) in Lemma \ref{rounding error of comp for inner step} and (\ref{compqd per err for q2}), we derive
\begin{equation}\label{compqd err for q2}
\begin{split}
|\epsilon\epsilon q_2^{(n)}|&\leq\bigg\{d_1^{(n)}\gamma_7\gamma_8+D_1\gamma_5\gamma_6(1+\gamma_7\gamma_8)\bigg(\frac{|\bar{e}_1^{(n+1)}|}{|e_1^{(n+1)}|}+1+\frac{|\bar{e}_1^{(n)}|}{|e_1^{(n)}|}\bigg)\bigg\}|q_2^{(n)}|\\
&\leq{D_1}\times\big\{\gamma_7\gamma_8+\gamma_5\gamma_6(1+\gamma_7\gamma_8)\big\}|\bar{q}_2^{(n)}|\\
&\leq{D_1}\gamma_{13}^2|\bar{q}_2^{(n)}|.
\end{split}
\end{equation}

\textbf{Step 3}: For $m=2$, due to $D_1\geq 1$,  we obtain in a similar way
\begin{equation}\label{compqd err for e2}
\begin{split}
|\epsilon\epsilon e_2^{(n)}|&\leq\big\{\gamma_3\gamma_4+{D_1}\gamma_{13}^2(1+\gamma_3\gamma_4)\big\}|\bar{e}_2^{(n)}|\\
&\leq{D_1}\times\big\{\gamma_3\gamma_4+\gamma_{13}^2(1+\gamma_3\gamma_4)\big\}|\bar{e}_2^{(n)}|\\
&\leq{D_1}\gamma_{16}\gamma_{17}|\bar{e}_2^{(n)}|,
\end{split}
\end{equation}
where $\bar{e}_2^{(n)}=|\bar{q}_2^{(n+1)}|+|\bar{q}_2^{(n)}|+|\bar{e}_1^{(n+1)}|$.

Again, for the next step, as $D_1\ll\frac{1}{u}$, from (\ref{compneq 9}) in Lemma \ref{rounding error of comp for inner step}, (\ref{compqd err for q2}) and (\ref{compqd err for e2}), we have
\begin{equation}\label{compqd per err for q3}
\begin{split}
|\bar{\beta}_3^{(n)}|&\leq{d_2^{(n)}}\times\bigg(D_1\gamma_{16}\gamma_{17}\frac{|\bar{e}_2^{(n+1)}|}{|e_2^{(n+1)}|}|q_3^{(n)}|+D_1\gamma_{13}^2\frac{|\bar{q}_{2}^{(n+1)}|}{|q_{2}^{(n+1)}|}|q_3^{(n)}|\\
& \qquad +D_1\gamma_{16}\gamma_{17}\frac{|\bar{e}_2^{(n)}|}{|e_2^{(n)}|}|q_3^{(n)}|+({D_1})^2\gamma_{13}^2\gamma_{16}\gamma_{17}\frac{|\bar{q}_2^{(n+1)}|}{|q_2^{(n+1)}|}\frac{|\bar{e}_2^{(n+1)}|}{|e_2^{(n+1)}|}|q_3^{(n)}|\bigg)\\
&\leq{D_1D_2}\gamma_{16}\gamma_{17}|\bar{q}_3^{(n)}|,
\end{split}
\end{equation}
where $D_2\geq 1$, $D_2\geq d_2^{(n)}$ and $\displaystyle{\bar{q}_3^{(n)}=\bigg(\frac{|\bar{e}_2^{(n+1)}|}{|e_2^{(n+1)}|}+\frac{|\bar{q}_2^{(n+1)}|}{|q_2^{(n+1)}|}+\frac{|\bar{e}_2^{(n)}|}{|e_2^{(n)}|}\bigg)|q_3^{(n)}|}$. Hence, from (\ref{per of compensated term}), (\ref{compneq 9}) in Lemma~\ref{rounding error of comp for inner step} and (\ref{compqd per err for q3}), we derive
\begin{equation*}\label{compqd err for q3}
\begin{split}
|\epsilon\epsilon q_3^{(n)}|&\leq{d_2^{(n)}}\gamma_7\gamma_8|q_3^{(n)}|+D_1D_2\gamma_{16}\gamma_{17}(1+\gamma_7\gamma_8)|\bar{q}_3^{(n)}|\\
&\leq{D_1D_2}\big\{\gamma_7\gamma_8+\gamma_{16}\gamma_{17}(1+\gamma_7\gamma_8)\big\}|\bar{q}_3^{(n)}|\\
&\leq{D_1D_2}\gamma_{24}^2|\bar{q}_3^{(n)}|.
\end{split}
\end{equation*}

\textbf{Step 4}:  In the general case, for $k \in \mathbb{N}$, and when $m=k$, we assume that
\begin{equation*}
|\widehat{e}_k^{(n)}-\widehat{\epsilon e}_k^{(n)}-e_k^{(n)}|=|\epsilon\epsilon e_k^{(n)}|\leq\bigg(\prod_{i=0}^{k-1}{D_i}\bigg)\gamma_{11k-5}\gamma_{11k-4}|\bar{e}_k^{(n)}|,
\end{equation*}
and
\begin{equation*}
|\widehat{q}_{k+1}^{(n)}-\widehat{\epsilon q}_{k+1}^{(n)}-q_{k+1}^{(n)}|=|\epsilon\epsilon q_{k+1}^{(n)}|\leq\bigg(\prod_{i=0}^{k}{D_i}\bigg)\gamma_{11k+2}^2|\bar{q}_{k+1}^{(n)}|,
\end{equation*}
where $d_0^{(n)*}=1$, $d_i^{(n)*}=\max\big\{\big|\frac{e_i^{(n)}}{e_i^{(n)}-\epsilon\epsilon e_i^{(n)}}\big|,1\big\}$, $D_i=\max\limits_n\big\{d_i^{(n)*}\big\}$, $\bar{e}_k^{(n)}$ and $\bar{q}_{k+1}^{(n)}$ are defined in (\ref{abs e}) and (\ref{abs q}), respectively.

Then, in a similar way, for $m=k+1$, according to (\ref{important equation of compqd}), (\ref{compneq 3}) in Lemma \ref{rounding error of comp for inner step}, (\ref{compqd err for em with 2}), (\ref{compqd err for q(m+1) with 2}) and (\ref{abs e}), we have
\begin{equation}\label{compqd err for em+1}
\begin{split}
|\widehat{e}_{k+1}^{(n)}-\widehat{\epsilon e}_{k+1}^{(n)}-e_{k+1}^{(n)}|&\leq\bigg\{\gamma_3\gamma_4+\bigg(\prod_{i=0}^{k}{D_i}\bigg)\gamma_{11k+2}^2(1+\gamma_3\gamma_4)\bigg\}|\bar{e}_{k+1}^{(n)}|\\
&\leq\bigg(\prod_{i=0}^{k}{D_i}\bigg)\gamma_{11(k+1)-5}\gamma_{11(k+1)-4}|\bar{e}_{k+1}^{(n)}|.
\end{split}
\end{equation}
Next, by considering $\prod\limits_{i=0}^{k}D_i\ll\frac{1}{u}$, from (\ref{compneq 9}) in Lemma \ref{rounding error of comp for inner step}, (\ref{compqd err for q(m+1) with 2}), (\ref{compqd err for em+1}) and (\ref{abs q}), we derive
\begin{equation}\label{compqd per err for qm+2}
\begin{split}
|\bar{\beta}_{k+2}^{(n)}|&\leq{d}_{k+1}^{(n)}\bigg(\prod_{i=0}^{k}{D_i}\bigg)\times\bigg(\gamma_{11k+5}\gamma_{11k+6}\frac{|\bar{e}_{k+1}^{(n+1)}|}{|e_{k+1}^{(n+1)}|}|q_{k+2}^{(n)}|+\gamma_{11k+2}^2\frac{|\bar{q}_{k+1}^{(n+1)}|}{|q_{k+1}^{(n+1)}|}|q_{k+2}^{(n)}|\\
& \qquad +\gamma_{11k+5}\gamma_{11k+6}\frac{|\bar{e}_{k+1}^{(n)}|}{|e_{k+1}^{(n)}|}|q_{k+2}^{(n)}|+\bigg(\prod_{i=0}^{k}{D_i}\bigg)\gamma_{11k+2}^2\gamma_{11k+5}\gamma_{11k+6}\frac{|\bar{q}_{k+1}^{(n+1)}|}{|q_{k+1}^{(n+1)}|}\frac{|\bar{e}_{k+1}^{(n+1)}|}{|e_{k+1}^{(n+1)}|}|q_{k+2}^{(n)}|\bigg)\\
&\leq{\bigg(\prod_{i=0}^{k+1}D_i\bigg)}\gamma_{11k+5}\gamma_{11k+6}|\bar{q}_{k+2}^{(n)}|,
\end{split}
\end{equation}
Hence, from (\ref{important equation of compqd}), (\ref{compneq 9}) in Lemma \ref{rounding error of comp for inner step} and (\ref{compqd per err for qm+2}), we finally obtain
\begin{equation*}\label{compqd err for qm+2}
\begin{split}
|\widehat{q}_{k+2}^{(n)}-\widehat{\epsilon q}_{k+2}^{(n)}-q_{k+2}^{(n)}|&\leq{d_{k+1}^{(n)}}\gamma_7\gamma_8|q_{k+2}^{(n)}|+\bigg(\prod_{i=0}^{k+1}D_i\bigg)\gamma_{11k+5}\gamma_{11k+6}(1+\gamma_7\gamma_8)|\bar{q}_{k+2}^{(n)}|\\
&\leq\bigg(\prod_{i=0}^{k+1}D_i\bigg)\big\{\gamma_7\gamma_8+\gamma_{11k+5}\gamma_{11k+6}(1+\gamma_7\gamma_8)\big\}|\bar{q}_{k+2}^{(n)}|\\
&\leq\bigg(\prod_{i=0}^{k+1}D_i\bigg)\gamma_{11(k+1)+2}^2|\bar{q}_{k+2}^{(n)}|.
\end{split}
\end{equation*}
And therefore, by induction we obtain the result.
\end{proof}

We remark that if $c_i\in\mathbb{F}$ for $i=0,1,2\ldots$ in Algorithm \ref{Algor compqd}, as described in (\ref{floating inputs of compqd}) in Section \ref{sec 4 comp}, the perturbations of inputs $\widehat{q}_1^{(n)}$ and $\widehat{\epsilon q}_1^{(n)}$ in {\tt Compqd} are slightly different. From Theorem \ref{EFT div dingli}, we have
\begin{equation*}
q_1^{(n)}=\frac{c_{n+1}}{c_n}=\widehat{q}_1^{(n)}+\frac{r}{c_n},
\end{equation*}
and
\begin{equation*}
\bigg|\frac{r}{c_n}\bigg|\leq{u}\bigg|\frac{c_{n+1}}{c_n}\bigg|=u|q_1^{(n)}|.
\end{equation*}
Thus, we obtain that
\begin{equation}\label{compqd err for q1def}
|\widehat{q}_1^{(n)}-\widehat{\epsilon q}_1^{(n)}-q_1^{(n)}|=\bigg|\frac{r}{c_n}-r\oslash{c_n}\bigg|\leq\gamma_1\bigg|\frac{r}{c_n}\bigg|\leq{u}\gamma_1|q_1^{(n)}|.
\end{equation}
In a similar way as the proof of Theorem \ref{rounding error bounds of comp qd}, using (\ref{compqd err for q1def}) instead of (\ref{compqd err for q1}), we can obtain the forward rounding error bounds of $\widehat{e}_m^{(n)}-\widehat{\epsilon e}_m^{(n)}$ and $\widehat{q}_{m+1}^{(n)}-\widehat{\epsilon q}_{m+1}^{(n)}$ in {\tt Compqd} as
\begin{equation*}\label{compqd err for em of floating-point coefficient}
|\widehat{e}_m^{(n)}-\widehat{\epsilon e}_m^{(n)}-e_m^{(n)}|\leq\bigg(\prod_{i=0}^{m-1}D_i\bigg)\times\gamma_{11m-5}\gamma_{11m-6}|\bar{e}_m^{(n)}|,
\end{equation*}
and
\begin{equation*}\label{compqd err for q(m+1) of floating-point coefficient}
|\widehat{q}_{m+1}^{(n)}-\widehat{\epsilon q}_{m+1}^{(n)}-q_{m+1}^{(n)}|\leq\bigg(\prod_{i=0}^{m}D_i\bigg)\times\gamma_{11m+1}^2|\bar{q}_{m+1}^{(n)}|.
\end{equation*}

Finally, we will give the forward rounding error bounds of $\widehat{e}_m^{(n)}$ and $\widehat{q}_{m+1}^{(n)}$ updated in {\tt Compqd} in Theorem~\ref{rounding error bounds of comp qd}.

\begin{theorem}\label{rounding error bounds of comp qd}
The forward rounding error bounds for the {\tt Compqd} algorithm, in the real coefficients case ($c_i\in\mathbb{R}$), are given by
\begin{equation*}\label{compqd err for em}
|\widehat{e}_m^{(n)}-e_m^{(n)}|\leq{u}|e_m^{(n)}|+\bigg(\prod_{i=0}^{m-1}D_i\bigg)\times\gamma_{11m-4}^2|\bar{e}_m^{(n)}|,
\end{equation*}
and
\begin{equation*}\label{compqd err for q(m+1)}
|\widehat{q}_{m+1}^{(n)}-q_{m+1}^{(n)}|\leq{u}|q_{m+1}^{(n)}|+\bigg(\prod_{i=0}^{m}D_i\bigg)\times\gamma_{11m+2}\gamma_{11m+3}|\bar{q}_{m+1}^{(n)}|,
\end{equation*}
with $D_i$ defined in (\ref{def D}) and supposing $\prod\limits_{i=1}^{m}D_i\ll\frac{1}{u}$, and where $\bar{e}_m^{(n)}$ and $\bar{q}_{m+1}^{(n)}$ are defined in Definition~\ref{condition number}.
\end{theorem}

\begin{proof}

In {\tt Compqd} algorithm, by using {\tt FastTwoSum} to update the result, from Theorem \ref{EFT dingli},  and (\ref{important equation of compqd}), we have
\begin{equation}\label{inequation e}
|\widehat{\epsilon e}_m^{(n)}|\leq{u}|\widehat{e}_m^{(n)}-\widehat{\epsilon e}_m^{(n)}|\leq{u}|e_m^{(n)}-\epsilon\epsilon e_m^{(n)}|\leq{u}(|e_m^{(n)}|+|\epsilon\epsilon{e}_m^{(n)}|).
\end{equation}
Thus, by (\ref{compqd err for em with 2}) in Theorem \ref{rounding error bounds of comp qd with 2} and (\ref{inequation e}), we obtain
\begin{equation*}
\begin{split}
|\widehat{e}_m^{(n)}-e_m^{(n)}|\leq&|\widehat{e}_m^{(n)}-\widehat{\epsilon{e}}_m^{(n)}-e_m^{(n)}|+|\widehat{\epsilon e}_m^{(n)}|\\
\leq&u|e_m^{(n)}|+(1+u)\bigg(\prod_{i=0}^{m-1}D_i\bigg)\times\gamma_{11m-5}\gamma_{11m-4}|\bar{e}_m^{(n)}|\\
\leq&u|e_m^{(n)}|+\bigg(\prod_{i=0}^{m-1}D_i\bigg)\times\gamma_{11m-4}^2|\bar{e}_m^{(n)}|.
\end{split}
\end{equation*}

Similarly, by (\ref{compqd err for q(m+1) with 2}) in Theorem \ref{rounding error bounds of comp qd with 2}, we obtain
\begin{equation*}
\begin{split}
|\widehat{q}_{m+1}^{(n)}-q_{m+1}^{(n)}|\leq&|\widehat{q}_{m+1}^{(n)}-\widehat{\epsilon{q}}_{m+1}^{(n)}-q_{m+1}^{(n)}|+|\widehat{\epsilon q}_{m+1}^{(n)}|\\
\leq&u|q_{m+1}^{(n)}|+(1+u)\bigg(\prod_{i=0}^{m}D_i\bigg)\times\gamma_{11m+2}^2|\bar{q}_{m+1}^{(n)}|\\
\leq&u|q_{m+1}^{(n)}|+\bigg(\prod_{i=0}^{m}D_i\bigg)\times\gamma_{11m+2}\gamma_{11m+3}|\bar{q}_{m+1}^{(n)}|.
\end{split}
\end{equation*}
\end{proof}

Therefore, from the condition numbers given in Definition~\ref{condition number}, we obtain directly the relative forward rounding error bounds of the {\tt Compqd} algorithm considering the perturbed inputs:
\begin{corollary}\label{relative error bounds of compqd}
The relative forward rounding error bounds for the {\tt Compqd} algorithm, in the real coefficients case ($c_i\in\mathbb{R}$), are given by
\begin{equation*}\label{qd err bound edef}
\frac{|\widehat{e}_m^{(n)}-e_m^{(n)}|}{|e_m^{(n)}|}\leq{u}+\Psi_{m-1} \, \gamma_{11m-4}^2 \, {\tt cond\_e_m^{(n)}} \equiv
{u}+\Psi_{m-1} \, \mathcal{O}(u^2) \, {\tt cond\_e_m^{(n)}},
\end{equation*}
and
\begin{equation*}\label{qd err bound qdef}
\frac{|\widehat{q}_{m+1}^{(n)}-q_{m+1}^{(n)}|}{|q_{m+1}^{(n)}|}\leq{u}+\Psi_m \, \gamma_{11m+2} \, \gamma_{11m+3} \, {\tt cond\_q_{m+1}^{(n)}} \equiv {u}+\Psi_m \, \mathcal{O}(u^2) \, {\tt cond\_q_{m+1}^{(n)}},
\end{equation*}
supposing
$\Psi_m=\prod\limits_{i=0}^{m}D_i\ll\frac{1}{u}$ and where $D_i$ is defined in (\ref{def D}).
\end{corollary}

In Corollary \ref{relative error bounds of compqd}, it should be noticed that ${\Psi}_{m-1}\leq {\Psi}_{m}$. Comparing with Corollary~\ref{relative error bounds of qd}, we remark that Corollary \ref{relative error bounds of compqd} states that the {\tt Compqd} algorithm is more accurate than the {\tt qd}. In fact, the effect of the compensated algorithm is
to multiply the condition numbers by $u^2$ instead by $u$ as in Corollary~\ref{relative error bounds of qd}. This fact permits, using {\tt Compqd}, to continue with the {\tt qd} table for more rows than using the standard algorithm as the instabilities given by the condition numbers will appear later. Note that, in opposite to the case of using double-double arithmetic, the rounding unit is always $u$ ($u^2$ in double-double approximately) but the stability is similar.

%-%--------------Numerical experiments------------------%%

\section{Numerical experiments}\label{test}
In this section we study the accuracy, performance and application of the proposed algorithm.
In Subsection~\ref{accuracy}, we compare the accuracy of {\tt qd}, {\tt compqd}  and {\tt DDqd} algorithms ({\tt qd} algorithm in double-double format based on the QD library \cite{Li02,Bailey,HLB}). Meanwhile, we  also present the error bounds of the {\tt qd} and {\tt Compqd} algorithms.
In Subsection~\ref{time compare}, the computational complexity of the above algorithms is given. We also show the performance of {\tt qd}, {\tt Compqd} and {\tt DDqd} in terms of running time.
In Subsection~\ref{application}, we give three simple applications to show the effectiveness of {\tt Compqd}, including the obtention of the coefficients of continued fractions from power series, the search of the poles of meromorphic functions and the zeros of polynomials.
Note that the working precision of our experiments is the standard IEEE-754 double precision. We use the Symbolic Toolbox in \textsc{Matlab} to obtain the `exact' results for comparisons.

\subsection{Accuracy}\label{accuracy}

Firstly, we consider, as test problems, several polynomials  of degree N-1 (with $N=10:7:500$) whose coefficients are random floating-point numbers uniformly distributed in the interval $(-1,1)$.

\begin{figure}[h!]
\begin{center}
\includegraphics[width=1.\textwidth]{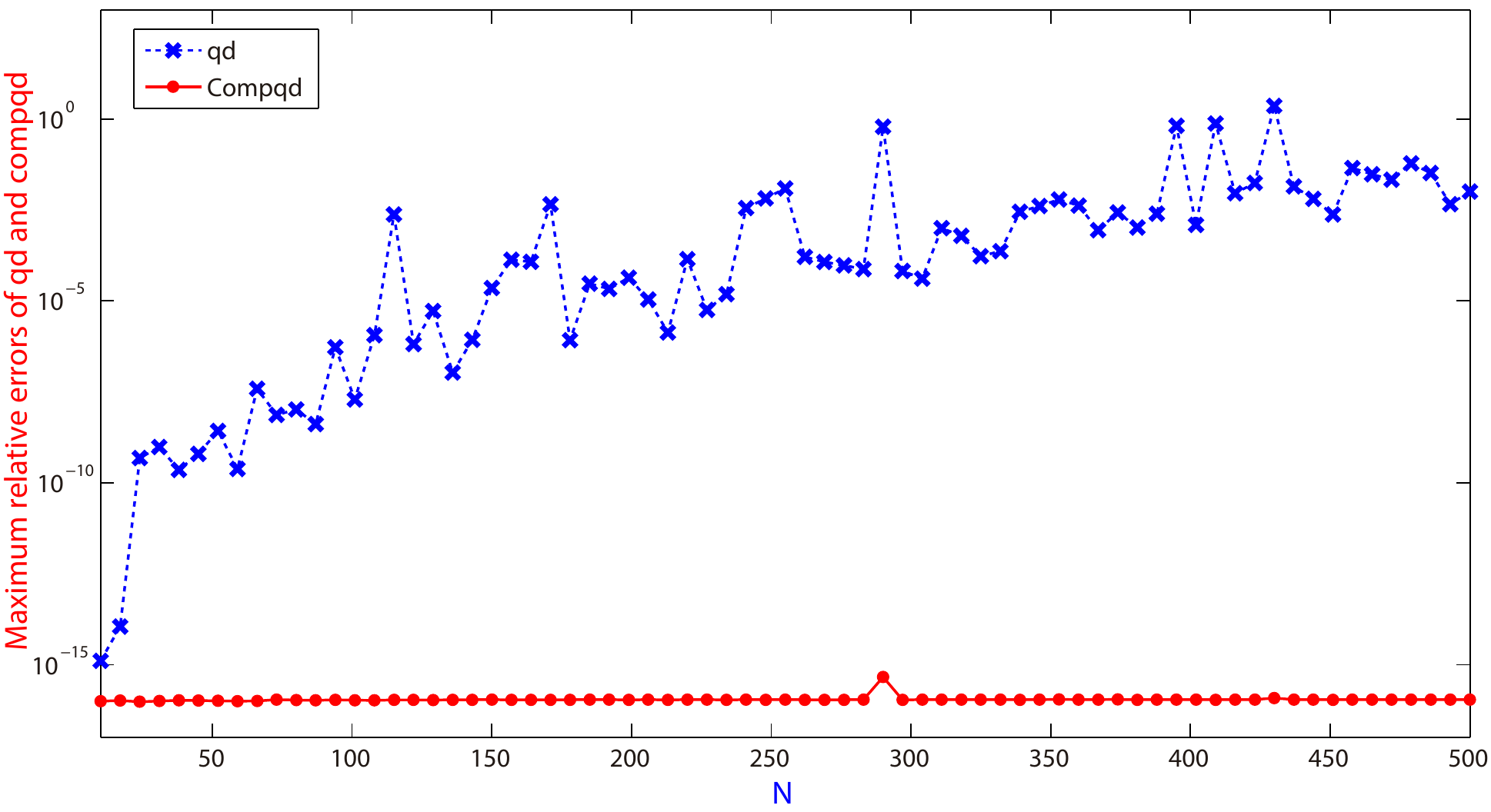}
\end{center}
\caption{Maximum relative errors on the $q$ columns ($\max\limits_{m}\max\limits_{n}\{|q_m^{(n)}-\widehat{q}_m^{(n)}|/|q_m^{(n)}|\}$)  using the {\tt qd} and {\tt Compqd} algorithms. }\label{errrand500}
\end{figure}

In Figure~\ref{errrand500} we consider the relative errors of all the terms $q_m^{(n)}$ for each polynomial of degree $N$.
On the vertical axis  we plot the maximum relative errors on the $q$ columns $\max\limits_{m}\max\limits_{n}\{|q_m^{(n)}-\widehat{q}_m^{(n)}|/|q_m^{(n)}|\}$ for $m=1:[{(N+1)}/{2}]$ and $n=0:N-2m+1$ for $N=10:7:500$ using the {\tt qd} and {\tt Compqd} algorithms.
As we can see, in some cases  the results of {\tt qd} algorithm have no significant digit, and
the relative errors of {\tt qd} increase  with $N$. However, the maximum relative errors of all $q_m^{(n)}$ for each polynomial of degree $N$ computed by the {\tt Compqd} algorithm are in all cases smaller than $10^{-15}$. Hence, {\tt Compqd} is much more stable than {\tt qd} with floating-point inputs.

Next, we consider the forward relative errors of {\tt qd} tables computed by using the {\tt qd}, {\tt Compqd} and {\tt DDqd} (Algorithm~\ref{Algor DDqd} in Appendix A) algorithms. In order to get the inputs of $q_1^{(n)}$, we consider the Taylor polynomial of degree 35 obtained by using the code {\tt Taylor(f(x),N)} in \textsc{Matlab} from the function
\begin{equation}\label{funtay}
\displaystyle{\frac{e^x}{(x-1)(x-2)(x+2)(x-3)}.}
\end{equation}
For accuracy, we approximate the real coefficients $c_i$ of the test polynomial by using double-double numbers $(c_i^{(h)},c_i^{(l)})$
from (\ref{coefficients 3 part}). We use \texttt{double($c_i$)} and \texttt{double($c_i$
-sym(double($c_i$)))} to represent $c_i^{(h)}$ and $c_i^{(l)}$ in \textsc{Matlab}, respectively. The condition numbers of computing $e_m^{(n)}$ and $q_m^{(n)}$ introduced in Definition \ref{condition number} verify the relationship shown in Figure \ref{condqdtable}.  Including the corresponding relative error bounds of {\tt qd} and {\tt Compqd} described in Corollary~\ref{relative error bounds of qd} and Corollary~\ref{relative error bounds of compqd}, the forward relative errors of {\tt qd}, {\tt Compqd} and {\tt DDqd} are reported in Figure~\ref{cond} for computing the terms $e_m^{(n)}$ and $q_{m+1}^{(n)}$, respectively. Here, $\Phi_{17}\approx 273.26$ is the largest $\Phi_{i}$ for $i=0,1,2\cdots,17$, and the same case happens to $\Psi_{17}\approx63.13$ in our numerical test, which means the terms $\Phi_{m}$ and $\Psi_{m}$, in Corollary~\ref{relative error bounds of qd} and Corollary~\ref{relative error bounds of compqd}, respectively, are reasonable in size.

\begin{figure}[h!]
\begin{center}
\includegraphics[width=1.\textwidth]{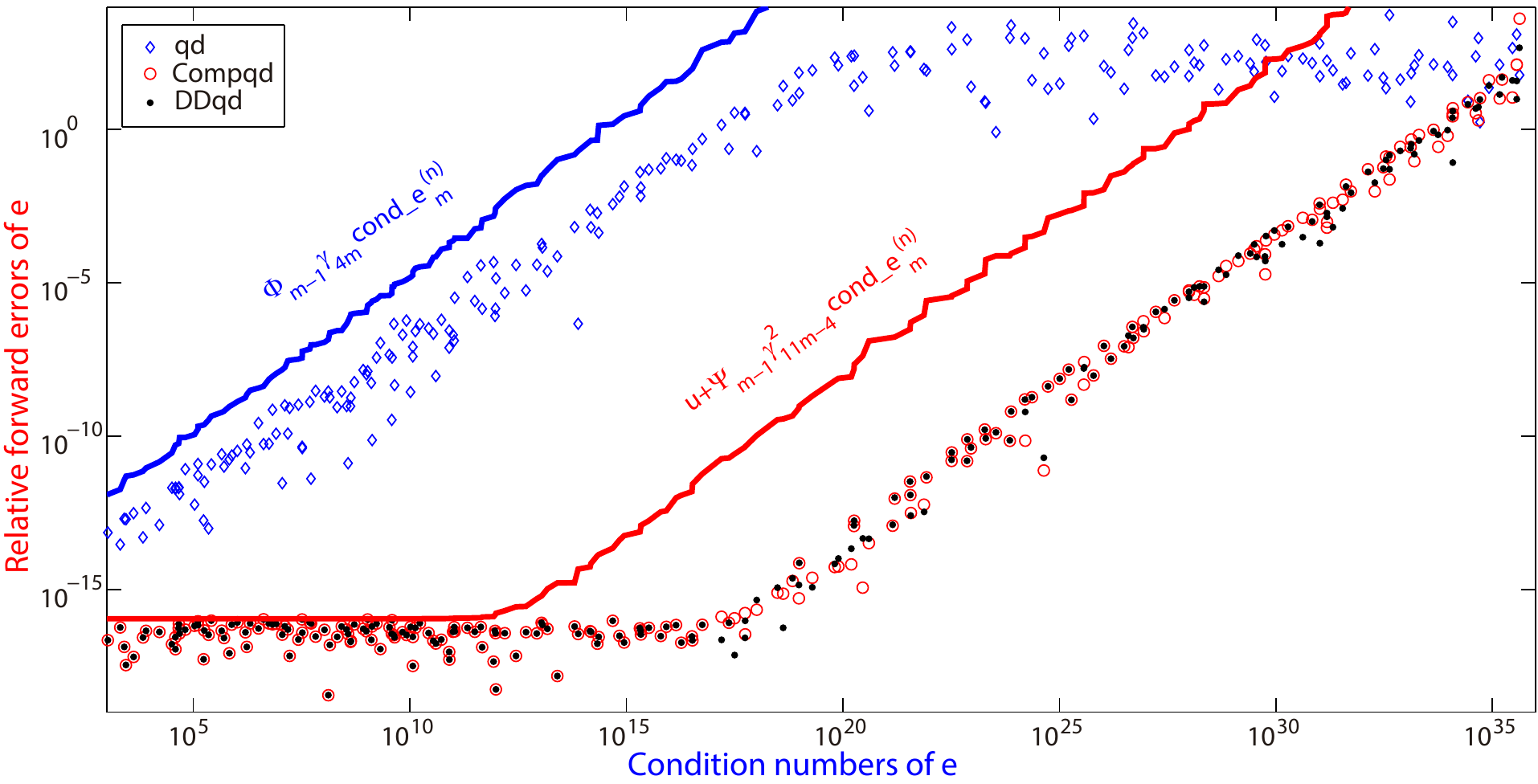}\\[0.5cm]
\includegraphics[width=1.\textwidth]{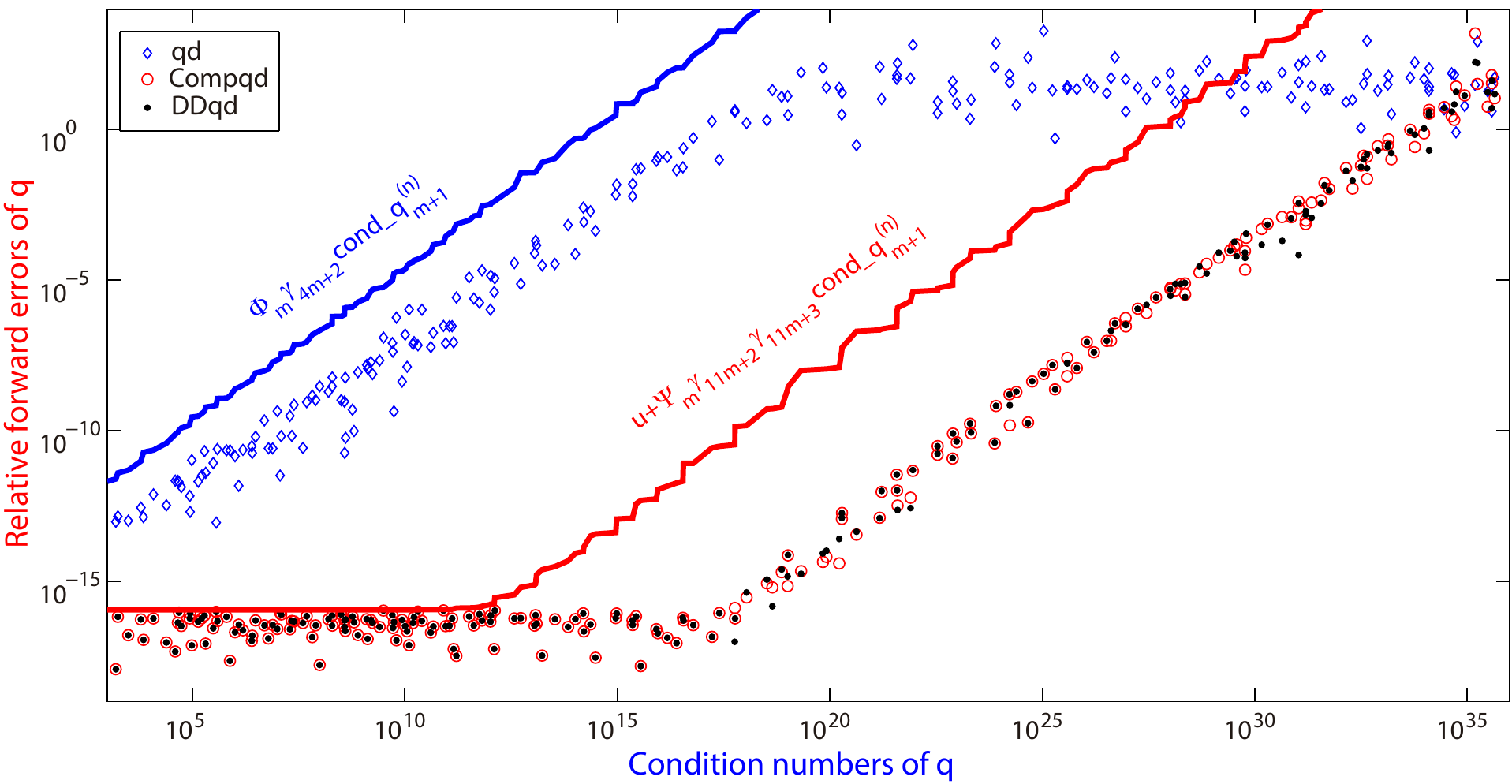}
\end{center}
\caption{Accuracy of terms $e_m^{(n)}$ and $q_{m+1}^{(n)}$ computed with the {\tt qd}, {\tt Compqd} and {\tt DDqd} algorithms with respect to the condition numbers: real errors (dotted points) and theoretical bounds (continuous lines).}\label{cond}
\end{figure}

In Figure \ref{cond}, we can observe that the {\tt qd} algorithm is unstable, and its relative error increases linearly (in logarithmic scale) when the condition number is smaller than ${1}/{u} \approx 10^{16}$.
As expected, {\tt Compqd} is more stable, and its relative errors  are equal to or smaller than the working precision $u$ when the condition number is smaller than ${1}/{u}$, and its relative error increases linearly when the condition numbers are between ${1}/{u}$ and ${1}/{u^2}$. Obviously, when $N$ is extremely large, {\tt Compqd} will not obtain accurate results, but it permits to compute accurately {\tt qd} tables of reasonable size. Moreover, {\tt DDqd} has almost the same accuracy as {\tt Compqd}. We remark the good agreement of the numerical tests and the theoretical bounds obtained in the previous sections (Corollary~\ref{relative error bounds of qd} and Corollary~\ref{relative error bounds of compqd}).

\subsection{Computational complexity and running time performance}\label{time compare}
Another important point is related with the CPU time. In this subsection, we show the computational complexity of the {\tt qd}, {\tt Compqd} and {\tt DDqd} algorithms, together with their practical performance in terms of running time.

Firstly, we assume that
the  initialization of the column $q_1^{(n)}$ requires
 ${\tt F_{input}}$ flops, and computing $e_m^{(n)}$ and $q_{m+1}^{(n)}$ by {\tt qd} scheme (\ref{qd table}) requires ${\tt F_e}$ and ${\tt F_q}$ flops in the inner loop (\ref{inner loop part}), respectively. It is easy to see that ${\tt F_e}={\tt F_q}=2$ and ${\tt F_{input}}=1$ flops in {\tt qd}. From Algorithms \ref{Algor compqd}, \ref{TwoSum}, \ref{FastTwoSum}, \ref{TwoProd}, \ref{TwoDiv} and \ref{Div_dd_dd} (see Appendix~A), we obtain that ${\tt F_e}=19$, ${\tt F_q}=50$ and ${\tt F_{input}}=100$ flops in {\tt Compqd}. Similarly, from Algorithms \ref{add_dd_dd}, \ref{prod_dd_dd}, \ref{Div_dd_dd} and~\ref{Algor DDqd} (see Appendix~A), we have ${\tt F_e}=40$, ${\tt F_q}=124$ and ${\tt F_{input}}=100$ flops in {\tt DDqd}. Then, the computational complexity of all the algorithms needed for computing $e_m^{(n)}$ and $q_{m+1}^{(n)}$ is $m^2{\tt F_e}+m(m-1){\tt F_q}+2m{\tt F_{input}}$ and $m^2{\tt F_q}+m(m+1){\tt F_e}+(2m+1){\tt F_{input}}$ flops, respectively. Hence, we can derive the total computational complexity of the {\tt qd}, {\tt Compqd} and {\tt DDqd} algorithms  for computing $e_m^{(n)}$:
\begin{itemize}
  \item {\tt qd}: $4m^2$ flops,
  \item {\tt Compqd}: $69m^2+150m$ flops,
  \item {\tt DDqd}: $164m^2+76m$ flops,
\end{itemize}
\noindent and for computing $q_{m+1}^{(n)}$:
\begin{itemize}
  \item {\tt qd}: $4m^2+4m+1$ flops,
  \item {\tt Compqd}: $69m^2+219m+100$ flops,
  \item {\tt DDqd}: $164m^2+240m+100$ flops.
\end{itemize}

We have measured the average ratios of the required flops of {\tt Compqd} and {\tt DDqd} over that of {\tt qd} for computing $e_m^{(n)}$ and $q_{m+1}^{(n)}$ for $m=50:5:1000$ in Table \ref{flops and times}. The ratios for computing $e_m^{(n)}$ and $q_{m+1}^{(n)}$ are almost the same. We can observe  that {\tt Compqd} has only 17 times the theoretical complexity of {\tt qd}, while {\tt DDqd} has 41 times. Therefore, although {\tt Compqd} algorithm has nearly the same accuracy as the {\tt DDqd} one, it requires on the average about 42.11\% of flop counts of {\tt DDqd}, and moreover, it is not required to use any different or extended arithmetic in the algorithm. This fact gives one of the greatest advantages of using the {\tt Compqd} algorithm, it has almost double-double precision without using extended arithmetic and it is much faster.

\begin{table}[ht]
\centering
\caption{Theoretical flop count ratios and measured running time ratios for computing $e_m^{(n)}$ and $q_{m+1}^{(n)}$.}\label{flops and times}
\begin{tabular}{c|cc|c}\hline
& ${\tt Compqd}/{\tt qd}$ & ${\tt DDqd}/{\tt qd}$ & ${\tt Compqd}/{\tt DDqd}$ \\\hline
Theoretical & 17.24 & 40.95 & 42.11\% \\
Measured & 3.33 & 12.11 & 27.53\% \\\hline
\end{tabular}
\end{table}

Besides, we compare {\tt qd}, {\tt Compqd} and {\tt DDqd} in terms of measured computing time. The tests are performed on a PC with a Intel(R) Core(TM) i7-4790 processor, with four cores each at 3.60Ghz and 4GB of main memory. The testing environment is under the \texttt{gcc} compiler, version 4.6.3, with the compiler option \texttt{-o2} on Linux Ubuntu 12.04. We generate the test polynomials with random coefficients in the interval $[-1,1]$, whose degrees vary from 50 to 1000 by the step of 5. The average time ratios of {\tt Compqd/qd} and {\tt DDqd/qd} for computing $e_m^{(n)}$ or $q_{m+1}^{(n)}$ are reported in Table \ref{flops and times}. Compared with the theoretical flop ratios, the measured running time ratios are obviously smaller than the theoretical flop count ones. This phenomenon is reasonable because the tested algorithms take benefit from the instruction level parallelism~(ILP) \cite{N.Louvet08, LL2} and the Fused-Multiply-and-Add instruction~(FMA) \cite{PM2000, YN}.
It is reasonable that {\tt Compqd} runs faster than {\tt DDqd} since compensated algorithms present more ILP than the algorithms computed in double-double arithmetic. In fact, now the increment in the measured time is just around 3.3 times, meaning that the
{\tt Compqd} algorithm provides a reasonable accurate version of the {\tt qd} algorithm.

\subsection{Applications}\label{application}
To illustrate the effectiveness and accuracy of {\tt Compqd} in more complex algorithms, we present its use in three simple but important applications.

\subsubsection{Computation of continued fractions}\label{continued fraction}
In literature \cite{CPVWJ} there are several algorithms developed to construct different continued fraction representations or approximations of functions.
The {\tt qd} algorithm constructs C-fractions from formal power series at $x = 0$. Note that a C-fraction is intimately connected with Pad\'e approximants, since its
successive approximants equal Pad\'e approximants on a staircase in the Pad\'e table.

Given a formal power series (\ref{FPS}), there exists precisely one corresponding continued fraction of the form (a regular C-fraction)
\begin{equation}\label{fraction}
f(z)= a_0 + \cfrac{a_1 z}{1+\cfrac{a_2 z}{1+\cfrac{a_3 z}{1+\cdots}}} \equiv a_0 + \vcenter{\hbox{\huge $\mathrm{K}$}}_{m=1}^\infty \left(\frac{a_m z}{1}\right),
\end{equation}
if and only if the Hankel determinants $H_m^{(1)}\neq0$ and $H_m^{(2)}\neq0$ for $m\in\mathbb{N}$ (see more details in \cite{CPVWJ,JT80}).
One algorithm shown  in Theorem \ref{fraction th} to obtain the coefficients of the regular C-fraction (\ref{fraction}) is based directly on the {\tt qd} scheme.
\begin{theorem}\label{fraction th}\cite{JT80}
Let  (\ref{FPS}) be the Taylor series at $z=0$ of a meromorphic function $f(z)$.
If $q_k^{(n)}$ and $e_k^{(n)}$ are the elements of the {\tt qd} scheme associated with $f(z)$, then the coefficients of the continued fraction (\ref{fraction}) corresponding to $f(z)$ are given by
\begin{equation*}\label{fraction relationship}
a_0=c_0, \quad a_1=c_1, \quad a_{2k}=-q_k^{(0)}, \quad a_{2k+1}=-e_k^{(0)}, \indent k\geq1.
\end{equation*}
\end{theorem}

\begin{figure}[h!]
\begin{center}
\includegraphics[width=1.\textwidth]{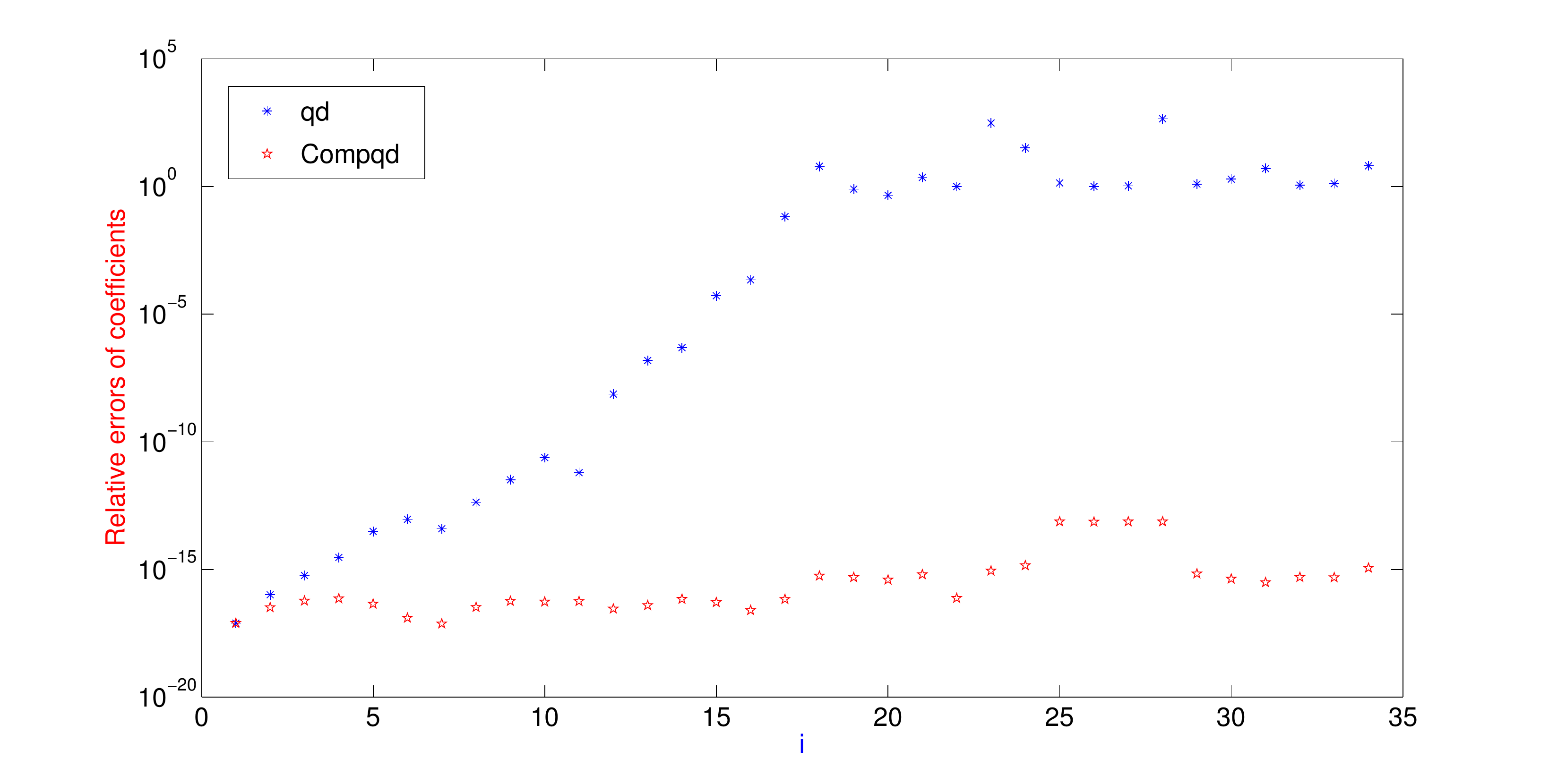}
\end{center}
\caption{Relative errors of the coefficients $a_i$ of the regular C-fraction computed using the {\tt qd} and {\tt Compqd} algorithms for the Taylor series of degrees 35 of the function (\ref{funtay}).}\label{Cfraction}
\end{figure}

We continue the second experiment of Subsection \ref{accuracy} by using again the Taylor polynomial of degree 35 with floating-point coefficients of the function (\ref{funtay}). Applying {\tt qd} and {\tt Compqd} algorithms to Theorem~\ref{fraction th}, we present the relative errors of the coefficients of the regular C-fraction in Figure \ref{Cfraction}. As we can see, the relative errors of the coefficients $a_i$ computed using {\tt qd} increase with $i$. However, the coefficients computed by {\tt Compqd} are much more accurate, and their maximum relative error is near to the rounding unit.

\subsubsection{Poles of meromorphic functions}\label{polessec}
The {\tt qd} algorithm can be used for the determination of poles of a meromorphic function $f(z)$ \cite{AC2010,rutishauser}. Considering the formal power series (\ref{FPS}) expansion of the function $f(z)$ with $z \in \mathbb{C}$, given the Hankel determinants (\ref{HDet}) associated with this series we say that the power series (\ref{FPS}) is ``$k$-normal'' if $H_m^{(n)} \not= 0$ for $m=0, 1, \ldots, k$ and $n \geq 0$. It is called ``ultimately $k$-normal'' if for every $0 \leq m \leq k$ there exists an $n(m)$ such that $H_m^{(n)} \not= 0$ for $n> n(m)$.
Classical results from complex analysis give basic algorithms for the location of poles in meromorphic functions:

\begin{theorem}\label{poles}\cite{Henrici,CUYT1997}
Let  (\ref{FPS}) be the Taylor series at $z=0$ of a meromorphic function $f(z)$ in the disk $B(0, R) = \{z \, : \, |z|<R\}$ and let the poles $z_i$ of $f \in B(0, R)$ be ordered such that
$$z_0 = 0 < |z_1| \leq |z_2| \leq \ldots < R,$$
each pole occurring as many times in the sequence $\{z_i\}_{i \in \mathbb{N}}$ as indicated by its order. If $f$ is ultimately $k$-normal for some integer $k>0$, then the {\tt qd} scheme associated with $f$ has the following properties:
\begin{itemize}
\item[(a)] For each $m$ with $0<m\leq k $ and $|z_{m-1}| < |z_m| < |z_{m+1}|$ where $z_0=0$ and $z_{m+1}=\infty$ if $f$ has only $m$ poles, we have $\lim\limits_{n \rightarrow \infty} q_m^{(n)} = z_m^{-1}$;
\item[(b)] For each $m$ with $0<m\leq k $ and $|z_m| < |z_{m+1}|$, we have $\lim\limits_{n \rightarrow \infty} e_m^{(n)} = 0$.
\end{itemize}
\end{theorem}
Therefore, there are three steps for locating poles by using Theorem \ref{poles}:
\begin{itemize}
\item \emph{Step 1:} Expand the function in Taylor series;
\item \emph{Step 2:} Obtain the complete {\tt qd} table from the coefficients of the Taylor series;
\item \emph{Step 3:} Use Theorem \ref{poles} to locate poles.
\end{itemize}

As test example, we consider the function
\begin{equation}\label{funtay2}
\displaystyle{\frac{e^x}{(x-1)(x-2)(x-3)(x-4)}.}
\end{equation}
 Firstly, we compute the Taylor series development (around $x=0$) of degree $N$ of the
  the function (\ref{funtay2}). In this test, we consider four degrees $N=24, 34, 44$ and $54$, and note that function (\ref{funtay2}) is obviously ultimately 4-normal. Then, we use {\tt qd} and {\tt Compqd} to obtain the complete {\tt qd} table. From Theorem \ref{poles}, the first pole computed by {\tt qd} in all the four cases are the same as those of {\tt Compqd} since they can be obtained by only one division operation. The other three poles are reported in Table \ref{poles compare 1}.
The {\tt qd} algorithm using exact arithmetic in the \textsc{Matlab} symbolic toolbox, {\tt Symqd}, based on Theorem \ref{poles} gives the results without any rounding error. We can observe that  {\tt Compqd} is more accurate than {\tt qd}, and its results are almost the same as those computed by {\tt Symqd}. We find that in the cases $N=34,44,54$ the last two poles obtained using {\tt qd}  have no significant digit, that means {\tt qd} can not deal accurately with the pole location. However, {\tt Compqd} can still get the poles.

We remark that when $N=54$ {\tt Compqd} can not find the last pole. To improve the performance of the algorithms in locating poles, we use the alternative method proposed in~\cite{CUYT1997}, and
we combine it with {\tt Compqd}  to find poles. Defining that any index $m$ such that the strict inequality $|z_m|<|z_{m+1}|$ holds is called a \emph{critical index}, the {\tt qd} scheme can determine the poles of a meromorphic function $f$ directly from its Taylor series using the methodology described in Theorem \ref{find poles}.

\begin{table}[ht]
\centering
\caption{The poles of Taylor polynomials of degree $N$  expanded from Eq. (\ref{funtay2}) obtained from Theorem \ref{poles} and using the {\tt qd}, {\tt Compqd} and {\tt Symqd} algorithms.}\label{poles compare 1}
\begin{tabular}{c|c|cccc}\hline
&Methods & Second pole & Third pole & Fourth pole \\\hline\hline
&{\tt qd} & 1.999360212952655 & 2.993928981359646 & 4.000508014082992 \\
N=24&{\tt Compqd} & 1.999360213958358 & 2.993916792495087 & 4.019757154976143 \\
&{\tt Symqd} & 1.999360213958358 & 2.993916792495087 & 4.019757154976143 &  \\\hline
&{\tt qd} & 1.999988303398561 & - & - \\
N=34&{\tt Compqd} & 1.999988805384870 & 2.999576789137349 & 4.001093405615016 \\
&{\tt Symqd} & 1.999988805384870 & 2.999576789137349 & 4.001093405610383 &  \\\hline
&{\tt qd} & 2.000958313366616 & - & - \\
N=44&{\tt Compqd} & 1.999999805766010 & 2.999974706425002 & 4.000063369400147 \\
&{\tt Symqd} & 1.999999805766010 & 2.999974706426370 & 4.000061511186811 &  \\\hline
&{\tt qd} & 1.999999999999998 & - & - \\
N=54&{\tt Compqd} & 1.999999996631584 & 2.999998762542696 & - \\
&{\tt Symqd} & 1.999999996631584 & 2.999998550118171 & 4.000003463711180 &  \\\hline
\end{tabular}
\end{table}

\begin{table}[ht]
\centering
\caption{The poles of Taylor polynomials of degree $N$ expanded from Eq. (\ref{funtay2}) obtained from Theorem \ref{find poles} and using the {\tt qd}, {\tt Compqd} and {\tt Symqd} algorithms.}\label{poles compare 2}
\begin{tabular}{c|c|cccc}\hline
&Methods & Second pole & Third pole & Fourth pole \\\hline\hline
&{\tt qd} & 1.999999109742843 & 2.999437417806726 & 4.002118134645123 \\
N=24&{\tt Compqd} & 1.999999129884058 & 2.999452305326862 & 4.001220145895098 \\
&{\tt Symqd} & 1.999999129884058 & 2.999452305326858 & 4.001220145895103 &  \\\hline
&{\tt qd} & 2.000000275935389 & 2.993218480452075 & - \\
N=34&{\tt Compqd} & 1.999999999984543 & 2.999999453378646 & 4.000001214856552 \\
&{\tt Symqd} & 1.999999999984540 & 2.999999453378657 & 4.000001214856524 &  \\\hline
&{\tt qd} & 1.999964072650627 & - & - \\
N=44&{\tt Compqd} & 2.000000000000001 & 2.999999999461029 & 4.000000079545716 \\
&{\tt Symqd} & 2.000000000000000 & 2.999999999465995 & 4.000000001186681 &  \\\hline
&{\tt qd} & - & - & - \\
N=54&{\tt Compqd} & 2.000000000000006 & 3.000000042940265 & 3.989674221270899 \\
&{\tt Symqd} & 2.000000000000000 & 2.999999999999479 & 4.000000000001159 &  \\\hline
\end{tabular}
\end{table}

\begin{theorem}\label{find poles}\cite{CUYT1997}
Let $m$ and $m+j$ with $j>1$ be two consecutive critical indices and let $f$ be $(m+j)$-normal. Let polynomials $p_k^{(n)}$ be defined by
\begin{equation*}
\begin{split}
p_0^{(n)}(z)=&1\\
p_{k+1}^{(n)}(z)=&zp_k^{(n+1)}(z)-q_{m+k+1}^{(n)}p_k^{(n)}(z),  \indent n\geq0,  \indent k=0,1,\cdots, j-1.
\end{split}
\end{equation*}
Then there exits a subsequence $\{n(\ell)\}_{\ell\in\mathbb{N}}$ such that
\begin{equation*}
\lim\limits_{\ell\longrightarrow \infty}p_j^{n(\ell)}(z)=(z-z_{m+1}^{-1})\cdots(z-z_{m+j}^{-1}).
\end{equation*}
\end{theorem}
There are four steps for locating poles by Theorem \ref{find poles}:
\begin{itemize}
\item \emph{Step 1:} Expand the function in Taylor series;
\item \emph{Step 2:} Obtain the incomplete {\tt qd} table (columns $q$) from the coefficients of Taylor series;
\item \emph{Step 3:} Use Theorem \ref{find poles} to generate a polynomial;
\item \emph{Step 4:} Solve the generated polynomial equation $p_j^{(n)}(z)=0$.
\end{itemize}

 Now, we consider again the Taylor polynomials of degree $N=24, 34, 44$ and $54$ expanded from function~(\ref{funtay2}), and we apply {\tt qd} and {\tt Compqd} algorithms to obtain the incomplete {\tt qd} table. From Theorem~\ref{find poles}, we derive that the first pole is $z_1 \simeq 1/q_1^{(N-1)}$.
 We consider the location of the other three poles and so $j=3$, $m=1$. It is obvious that function  (\ref{funtay2}) is  4-normal. Then, from Theorem \ref{find poles}, we can generate a polynomial
  \begin{equation*}
 \begin{split}
 p_3^{(N-9)}(z)=&z^3-\big(q_2^{(N-7)}+q_3^{(N-8)}+q_4^{(N-9)}\big)z^2+\big(q_2^{(N-8)}q_3^{(N-8)}+q_2^{(N-8)}q_4^{(N-9)}+\\
 &  q_3^{(N-9)}q_4^{(N-9)}\big)z-q_2^{(N-9)}q_3^{(N-9)}q_4^{(N-9)},
 \end{split}
 \end{equation*}
such that the reciprocals of the zeros of this polynomial are the three poles of (\ref{funtay2}).
Here, we use the code \texttt{solve(f)} in \textsc{Matlab} to find the zeros of $p_3^{(N-9)}(z)$. The first pole computed by {\tt qd} and {\tt Compqd} is the same, the other three poles are presented in Table \ref{poles compare 2}. As we can see, the results in Table~\ref{poles compare 2} are more accurate than in Table~\ref{poles compare 1}. Moreover, the three poles computed by using {\tt Compqd} are similar to the poles obtained by {\tt Symqd}. In the case $N=54$, the results by using {\tt qd} have no significant digit, but {\tt Compqd} can still keep some accuracy.

\subsubsection{Zeros of polynomials}\label{zeros}
The {\tt qd} algorithm can be used to find simultaneously all the zeros of a polynomial with real coefficients~\cite{HW1965}.
We consider the formal power series (\ref{FPS}) of degree $k$. Its zeros $z_m$ $(m=1,2,\dots,k)$ can be found as the poles of the rational function $r(z)=f(z)^{-1}$. From Theorem \ref{poles}, if the moduli of the zeros of $f(z)$ are all different, then the $m$-th $q$-column of $r(z)$ tends to $z_m^{-1}$ when the $m$-th $e$-column tends to zero. Let $f^*(z)=z^kf(z^{-1})$, then considering $r^*(z)=f^*(z{)}^{-1}$, $q$-columns of $r^*(z)$ tend to the reciprocals of the zeros of $f^*(z)$, which are the zeros of $f(z)$.

The \emph{progressive} form of {\tt qd} scheme \cite{rutishauser}, which is more suitable for this problem, can be used to find zeros. For a current detailed analysis of this algorithm and several modifications see \cite{AC2010}. The progressive {\tt qd} algorithm ({\tt proqd}) and its compensated algorithm ({\tt Compproqd}), are presented in Appendix~B. The {\tt qd} table of {\tt proqd} is built as follows.
\begin{equation*}\label{pro qd table}
\begin{array}{ccccccccc}
     &q_1^{(0)} & & q_2^{(-1)}& & q_3^{(-2)} & &\cdots &  \\
     0 & &e_1^{(0)} & & e_2^{(-1)}&  &\cdots  & & 0\\
     &q_1^{(1)} & & q_2^{(0)}& & q_3^{(-1)} & &\cdots & \\
     0 & &e_1^{(1)} & & e_2^{(0)}& & \cdots  & & 0\\
      &\vdots & &\vdots &\cdots &  & &  \\
     \vdots & & \vdots& \cdots& &  & &  \\
  \end{array}
\end{equation*}

In the numerical test, we consider the Laguerre orthogonal polynomial (see \cite{Szego}) of degree 35 defined by the three-term recurrence relation
\begin{equation*}\label{laguerre}
\begin{cases}
L_0(x)=1,\\
L_1(x)=1-x,\\
L_{k+1}(x)=\left(-\dfrac{1}{k+1}x+\dfrac{2k+1}{k+1}\right)L_k(x)-\dfrac{k}{k+1}L_{k-1}(x) & (k=1,2,3\cdots).
\end{cases}
\end{equation*}

\begin{figure}[h!]
\begin{center}
\includegraphics[width=1.\textwidth]{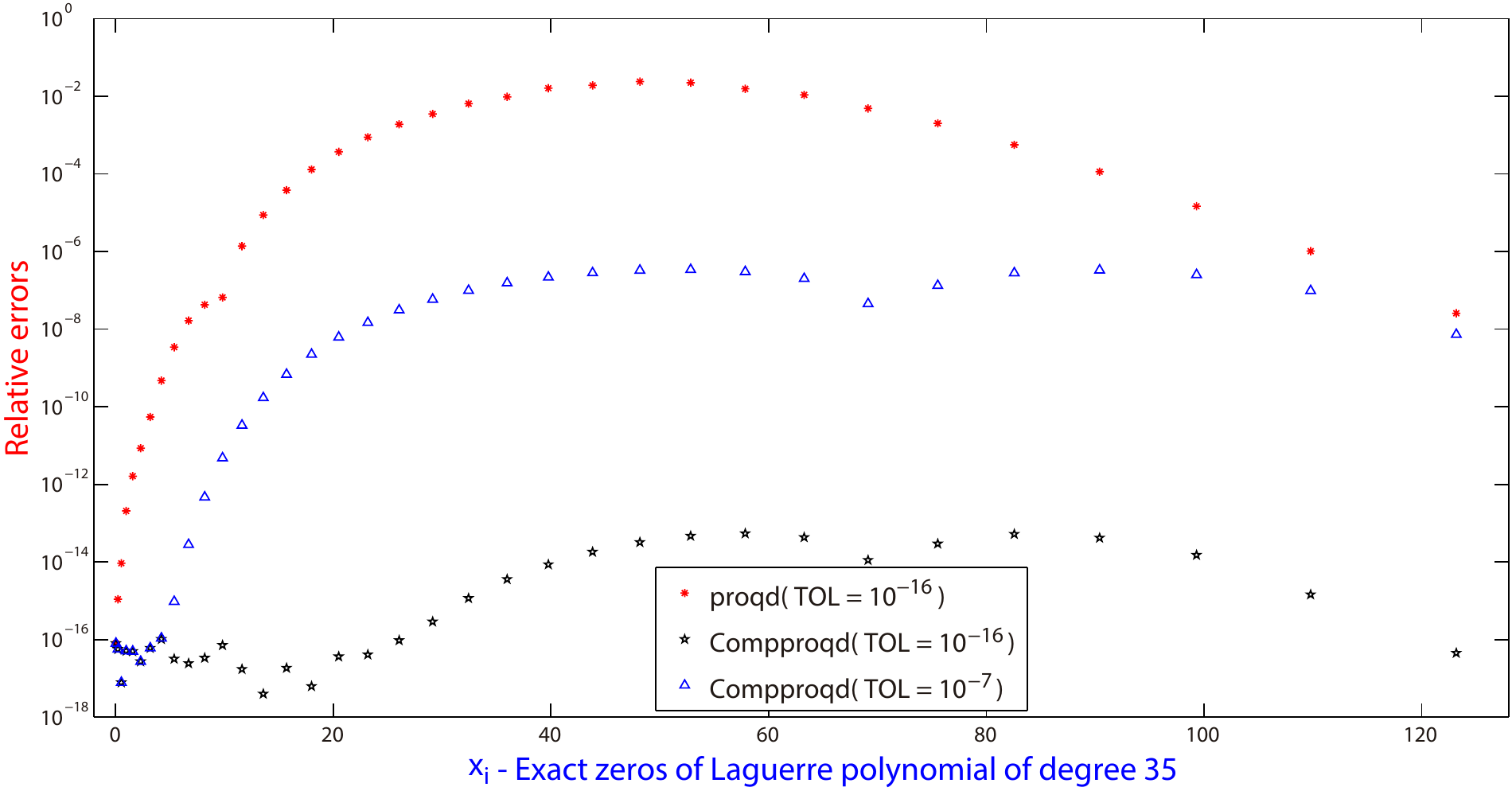}
\end{center}
\caption{Relative errors of the zeros of the Laguerre polynomial of degree 35 computed by using {\tt proqd} ($\texttt{TOL}=10^{-16}$) and {\tt Compproqd} ($\texttt{TOL}=10^{-7}$ and $\texttt{TOL}=10^{-16}$). On the horizontal axis we show the position $x_i$ of the zeros.}\label{zeros compare}
\end{figure}

We apply {\tt proqd} and {\tt Compproqd} with $\texttt{TOL}=10^{-16}$  to locate the zeros (\texttt{TOL} is the error tolerance to stop the iterative process of the algorithm).
For comparison, in the symbolic method, we use $\texttt{TOL}=10^{-33}$. The relative errors of zeros computed by {\tt proqd} and {\tt Compproqd} are reported in Figure \ref{zeros compare}. We observe that the relative errors of zeros computed by using {\tt proqd} are larger than those computed by using {\tt Compproqd}. That is to say, the {\tt Compproqd} algorithm is more stable and its relative error results are close to the rounding unit. Besides, we  test {\tt Compproqd} with $\texttt{TOL}=10^{-7}$, which requires a much smaller number of iterations  than  {\tt Compproqd} and {\tt proqd} with $\texttt{TOL}=10^{-16}$. Moreover, {\tt Compproqd} with $\texttt{TOL}=10^{-7}$ even runs faster than {\tt proqd} with $\texttt{TOL}=10^{-16}$. Note that we can also obtain the zeros with the required accuracy (e.g. half of the working precision) faster, just by fixing a smaller error tolerance $\texttt{TOL}$. We remark that this method can be combined with a Newton method to refine the approximate zeros, once we have a good initial data for the Newton process obtained from the {\tt Compproqd} algorithm.

%-%-----------------------------------------------------------------------Conclusions-------------------------------------------------------------------------------------%%
\section{Conclusions}
In this paper, we have studied in detail the quotient-difference (\texttt{qd}) algorithm, giving a complete analysis of its stability by providing forward rounding error bounds. In the  error analysis we have introduced new condition numbers adapted to the problem. Although it is well-known that the quotient-difference algorithm can be unstable, the theoretical bounds shown in this paper provide with a rigorous theoretical statement. Instead of using high-precision arithmetic or exact (symbolic) arithmetic to overcome this problem, as recommended in numerous papers, we introduce a new more accurate algorithm, the compensated quotient-difference ({\tt Compqd}) algorithm based on error-free transformations. This new algorithm can yield, in most cases, a full precision accuracy in working precision. The stability of the new method is studied and the forward rounding error bounds show that the effect of the compensated algorithm is to multiply the condition numbers by the square of the rounding unit, delaying  significantly the appearance of instability problems in standard precision. The advantages of the compensated quotient-difference algorithm are shown in several examples and in three practical applications: in the obtention of continued fractions and in pole and zero detection.

%-%--------Appendix------

\section*{Appendix A}\label{app}
The QD (quad-double) package \cite{Bailey,HLB} is based on the following algorithms:

\begin{algor}\cite{Knuth98}\label{TwoSum} Error-free transformation of the sum of two floating-point numbers\\
\indent~~~~function $[x, y] = {\tt TwoSum}(a, b)$\\
\indent~~~~~~~~$x = a \oplus b$\\
\indent~~~~~~~~$z = x \ominus a$\\
\indent~~~~~~~~$y= (a \ominus (x \ominus z)) \oplus (b \ominus z)$
\end{algor}
Algorithm \ref{TwoSum} requires $6$ flops.

\begin{algor}\cite{Dekker71}\label{FastTwoSum} Error-free transformation of the sum of two floating-point numbers ($|a|\geq |b|$)\\
\indent~~~~function $[x, y] = {\tt FastTwoSum}(a, b)$\\
\indent~~~~~~~~$x = a \oplus b$\\
\indent~~~~~~~~$y= (a \ominus x) \oplus b$
\end{algor}
Algorithm \ref{FastTwoSum} requires $3$ flops.

\begin{algor}\cite{Dekker71}\label{Split} Error-free split of a floating-point numbers into two parts \\
\indent~~~~function $[x, y] = {\tt Split}(a)$\\
\indent~~~~~~~~$c = {\tt factor} \otimes a$ \quad (in double precision {\tt factor} = $2^{27}+1$)\\
\indent~~~~~~~~$x = c \ominus (c \ominus a)$\\
\indent~~~~~~~~$y = a \ominus x$
\end{algor}
Algorithm \ref{Split} requires $4$ flops.

\begin{algor}\cite{Dekker71}\label{TwoProd} Error-free transformation of the product of two floating-point numbers\\
\indent~~~~function $[x, y] = {\tt TwoProd}(a, b)$\\
\indent~~~~~~~~$x = a \otimes b$\\
\indent~~~~~~~~{[$a1$, $a2$]} = {\tt Split}($a$)\\
\indent~~~~~~~~{[$b1$, $b2$]} = {\tt Split}($b$)\\
\indent~~~~~~~~$y = a2 \otimes b2 \ominus (((x \ominus a1 \otimes b1) \ominus a2 \otimes b1) \ominus a1 \otimes b2)$
\end{algor}
Algorithm \ref{TwoProd} requires $17$ flops.

\begin{algor}\cite{N.Louvet08,M.Pichat93}\label{TwoDiv} Error-free transformation of the division of two floating-point numbers\\
\indent~~~~function $[q, r ] = {\tt DivRem}(a, b)$\\
\indent~~~~~~~~$q = a\oslash b$\\
\indent~~~~~~~~$[x, y] = {\tt TwoProd}(q, b)$\\
\indent~~~~~~~~$r = (a\ominus x)\ominus y$
\end{algor}
Algorithm \ref{TwoDiv} requires $20$ flops.

\begin{algor}\cite{HLB,N.Louvet08}\label{add_dd_d} Addition of a double-double number and a double
number\\
\indent~~~~function $[rh,rl] = \texttt{add\_dd\_d}(ah, al, b)$\\
\indent~~~~~~~~$[th,tl] = \texttt{TwoSum} (ah, b)$\\
\indent~~~~~~~~$tl = al\oplus tl$\\
\indent~~~~~~~~$[rh,rl] = \texttt{FastTwoSum}(th,tl)$
\end{algor}
Algorithm \ref{add_dd_d} requires $10$ flops.

\begin{algor}\cite{HLB,Li02}\label{add_dd_dd} Addition of a double-double number and double-double
number\\
\indent~~~~function $[ rh, rl ] = \texttt{add\_dd\_dd} (ah, al, bh,bl)$\\
\indent~~~~~~~~$[sh,sl] = \texttt{TwoSum} (ah, bh)$\\
\indent~~~~~~~~$[th,tl] = \texttt{TwoSum} (al, bl)$\\
\indent~~~~~~~~$sl=sl\oplus th$\\
\indent~~~~~~~~$th= sh\oplus sl$\\
\indent~~~~~~~~$sl= sl\ominus(th\ominus sh)$\\
\indent~~~~~~~~$tl= tl\oplus sl$\\
\indent~~~~~~~~$[rh,rl] = \texttt{FastTwoSum}(th,tl)$
\end{algor}
Algorithm \ref{add_dd_dd} requires $20$ flops.

\begin{algor}\cite{HLB,N.Louvet08}\label{prod_dd_d} Multiplication of a double-double number by a double
number\\
\indent~~~~function $[rh,rl] = \texttt{prod\_dd\_d}(ah, al, b)$\\
\indent~~~~~~~~$[th,tl] = \texttt{TwoProd} (ah, b)$\\
\indent~~~~~~~~$tl = al\otimes b\oplus tl$\\
\indent~~~~~~~~$[rh,rl] = \texttt{FastTwoSum}(th,tl)$
\end{algor}
Algorithm \ref{prod_dd_d} requires $22$ flops.

\begin{algor}\cite{HLB,GLL1}\label{prod_dd_dd} Multiplication of two double-double numbers\\
\indent~~~~function $[rh,rl] = \texttt{prod\_dd\_dd}(ah, al, bh,bl)$\\
\indent~~~~~~~~$[th,tl] = \texttt{TwoProd} (ah, bh)$\\
\indent~~~~~~~~$ tl=(ah\otimes bl)\oplus (al\otimes bh)\oplus tl$\\
\indent~~~~~~~~$[rh,rl] = \texttt{FastTwoSum}(th,tl)$
\end{algor}
Algorithm \ref{prod_dd_dd} requires $24$ flops.

\begin{algor}\cite{Bailey}\label{Div_dd_dd} Division of two double-double numbers\\
\indent~~~~function $[rh, rl] = {\tt Div\_dd\_dd}(ah,al,bh,bl)$\\
\indent~~~~~~~~$q_1=ah/bh$\\
\indent~~~~~~~~$[th,tl]={\tt prod\_dd\_d}(bh,bl,q_1)$\\
\indent~~~~~~~~$[rh,rl]={\tt add\_dd\_dd} (ah, al, -th,-tl)$\\
\indent~~~~~~~~$q_2=rh/bh$\\
\indent~~~~~~~~$[th,tl]={\tt prod\_dd\_d}(bh,bl,q_2)$\\
\indent~~~~~~~~$[rh,rl]={\tt add\_dd\_dd} (rh, rl, -th,-tl)$\\
\indent~~~~~~~~$q_3=rh/bh$\\
\indent~~~~~~~~$[q_1, q_2] ={\tt FastTwoSum}(q_1, q_2)$\\
\indent~~~~~~~~$[rh, rl] = {\tt add\_dd\_d}(q_1, q_2, q_3)$
\end{algor}
Algorithm \ref{Div_dd_dd} requires $100$ flops.
\medskip

\noindent The double-double arithmetic version of the {\tt qd} algorithm, is the {\tt DDqd} algorithm, and it is given by
\begin{algor}{\tt DDqd} algorithm \label{Algor DDqd} ({\tt qd} algorithm in double-double arithmetic)\\
\indent~~{\bf input:}~~{${eh}_0^{(n)}=0$, $el_0^{(n)}=0$,  $n=1,2,...$ \\
\indent\indent~~~~~~~~$[{qh}_1^{(n)},ql_1^{(n)}]={\tt Div\_dd\_dd}(c^{(h)}_{n+1},c^{(l)}_{n+1},c^{(h)}_n,c^{(l)}_n)$, $n=0,1,...$}\\
\indent~~{\bf output:}~~{qd scheme}\\
\indent~~{\bf for} ~{$m= 1, 2, ...$}\\
\indent~~\indent{\bf for} {$n= 0, 1, ...$}\\
 \indent~~\indent~~~~    $[ rh, rl ] = \texttt{add\_dd\_dd} (qh_m^{(n+1)}, ql_m^{(n+1)}, -qh_m^{(n)},-ql_m^{(n)})$\\
 \indent~~\indent~~~~    $[ eh_m^{(n)}, el_m^{(n)} ] = \texttt{add\_dd\_dd} ( rh, rl,eh_{m-1}^{(n+1)},el_{m-1}^{(n+1)})$\\
  %$q_{m+1}^{(n)}=\frac{e_{m}^{(n+1)}}{e_{m}^{(n)}}q_{m}^{(n+1)}$
 \indent~~\indent~~~~    $[th, tl] = {\tt Div\_dd\_dd}(eh_{m}^{(n+1)},el_{m}^{(n+1)},eh_{m}^{(n)},el_{m}^{(n)})$\\
 \indent~~\indent~~~~    $[qh_{m+1}^{(n)}, ql_{m+1}^{(n)}] = {\tt prod\_dd\_dd}(th, tl,qh_{m}^{(n+1)},ql_{m}^{(n+1)})$\\
\indent~~\indent{\bf end}\\
\indent~~{\bf end}
\end{algor}
Algorithm \ref{Algor DDqd} requires $164$ flops in the inner loop.

\section*{Appendix B}
The \emph{progressive} form of the {\tt qd} scheme \cite{rutishauser} is given by:\\

\begin{algor}{\tt proqd} algorithm \label{Algor proqd} (The \emph{progressive} form of {\tt qd} algorithm)\\
\indent~~{\bf input:}~~{$q_m^{(-m+1)}=0$, $m=2,3,...$; $q_1^{(0)}=-\displaystyle{\frac{b_1}{b_0}}$,\\
\indent\indent~~~~~~~~$e_m^{(-m+1)}=\displaystyle{{b_{m+1}}/{b_m}}$, $e_0^{(m)}=0$, $e_k^{(m-k)}=0$, $m=1,2,...$}\\
\indent\indent~~~~~~~~\texttt{TOL} (error tolerance)\\
\indent~~{\bf output:}~~{qd scheme}\\
\indent~~{\bf for} {$m+n=  1, 2,  ...$}\\
\indent~~\indent{\bf for} {$m= 1, 2, ...$}\\
 \indent~~\indent~~~~  $q_m^{(n+1)}=e_m^{(n)}-e_{m-1}^{(n+1)}+q_{m}^{(n)}$\\
 \indent~~\indent~~~~  $e_{m}^{(n+1)}=\displaystyle{({q_{m+1}^{(n)}}/{q_{m}^{(n+1)}}}) \times e_{m}^{(n)}$\\
\indent~~\indent{\bf end}\\
\indent~~\indent{\bf if}\\
\indent~~\indent~~~~ $\max\limits_{m+n}\{e_m^{(n+1)}\}\leq \texttt{TOL}$, {\bf break}\\
\indent~~\indent{\bf end}\\
\indent~~{\bf end}
\end{algor}

The new compensated version of the \emph{progressive} form of the {\tt qd} scheme is given by:\\

\begin{algor}{\tt Compproqd} algorithm \label{Algor Compproqd} (The compensated {\tt proqd} algorithm)\\
\indent~~{\bf input:}~~{$\widehat{q}_m^{(-m+1)}=0$, $\widehat{\epsilon q}_m^{(-m+1)}=0$, $\widehat{e}_0^{(m-1)}=0$, $\widehat{\epsilon e}_0^{(m-1)}=0$, $m=2,3,...$; \\
\indent\indent~~~~~~~~$[\widehat{q}_1^{(0)},-\widehat{\epsilon q}_1^{(0)}]={\tt Div\_dd\_dd}(-b^{(h)}_{1},-b^{(l)}_{1},b^{(h)}_0,b^{(l)}_0)$\\
\indent\indent~~~~~~~~$[\widehat{e}_m^{(-m+1)},-\widehat{\epsilon e}_m^{(-m+1)}]={\tt Div\_dd\_dd}(b^{(h)}_{m+1},b^{(l)}_{m+1},b^{(h)}_m,b^{(l)}_m)$, $\widehat{e}_k^{(m-k)}=0$, $\widehat{\epsilon e}_k^{(m-k)}=0$, $m=1,2,...$}\\
\indent\indent~~~~~~~~\texttt{TOL} (error tolerance)\\
\indent~~{\bf output:}~~{qd scheme}\\
\indent~~{\bf for} ~{$m+n=  1, 2, ...$}\\
\indent~~\indent{\bf for} {$m= 1, 2, ...$}\\
  \indent~~\indent~~~~  $[s,\mu_1]={\tt TwoSum}(\widehat{e}_m^{(n)},-\widehat{e}_{m-1}^{(n+1)})$\\
 \indent~~\indent~~~~  $[\widehat{q}_m^{(n+1)},\mu_2]={\tt TwoSum}(s,\widehat{q}_{m}^{(n)})$\\
 \indent~~\indent~~~~  $ \widehat{\epsilon q}_m^{(n+1)}=\mu_1\oplus\mu_2\oplus\widehat{\epsilon e}_m^{(n)}\ominus\widehat{\epsilon e}_{m-1}^{(n+1)}\oplus\widehat{\epsilon q}_{m}^{(n)}$\\
  \indent~~\indent~~~~ $[\widehat{q}_m^{(n+1)},-\widehat{\epsilon q}_m^{(n+1)}]={\tt FastTwoSum}(\widehat{q}_m^{(n+1)},-\widehat{\epsilon q}_m^{(n+1)})$\\
 \indent~~\indent~~~~ $[t,\mu_3]={\tt DivRem}(\widehat{q}_{m+1}^{(n)},\widehat{q}_{m}^{(n+1)})$\\
  \indent~~\indent~~~~ $[\widehat{e}_{m}^{(n+1)},\mu_4]={\tt TwoProd}(t,\widehat{e}_{m}^{(n)})$\\
  \indent~~\indent~~~~ $ \widehat{\epsilon e}_{m}^{(n+1)}=(\mu_3\otimes\widehat{e}_m^{(n)}\oplus\mu_4\otimes\widehat{q}_m^{(n+1)}\oplus\widehat{\epsilon e}_m^{(n)}\otimes \widehat{q}_{m+1}^{(n)}\oplus\widehat{\epsilon q}_{m+1}^{(n)}\otimes \widehat{e}_m^{(n)}\ominus\widehat{\epsilon q}_m^{(n+1)}\otimes\widehat{e}_{m}^{(n+1)})\oslash \widehat{q}_m^{(n+1)}$\\
  \indent~~\indent~~~~ $[\widehat{e}_{m}^{(n+1)},-\widehat{\epsilon e}_{m}^{(n+1)}]={\tt FastTwoSum}(\widehat{e}_{m}^{(n+1)},-\widehat{\epsilon e}_{m}^{(n+1)})$\\
\indent~~\indent{\bf end}\\
\indent~~\indent{\bf if}\\
\indent~~\indent~~~~ $\max\limits_{m+n}\{\widehat{e}_m^{(n+1)}\}\leq \texttt{TOL}$, {\bf break}\\
\indent~~\indent{\bf end}\\
\indent~~{\bf end}
\end{algor}

%-%-------References-----------------

\section*{References}
\bibliographystyle{elsarticle-num}
%\bibliography{Compqd-v2}

\end{document}